\documentclass[hidelinks]{amsart}

\usepackage{amsmath}
\usepackage{amsthm}
\usepackage{orcidlink}
\usepackage{amssymb}
\usepackage{hyperref}
\usepackage{xcolor}
\usepackage{footmisc}
\usepackage{orcidlink}
\usepackage{xfrac}

\def\R{{\mathbf R}}

\def\d{{\partial}}
\def\eps{\varepsilon}

\theoremstyle{plain}
\newtheorem{theorem}{Theorem}[section]

\newtheorem{lemma}[theorem]{Lemma}

\newtheorem{proposition}[theorem]{Proposition}

\theoremstyle{remark}
\newtheorem{remark}[theorem]{Remark}

\numberwithin{equation}{section}

\title[Stability of dispersive boundary layers]{Stability of dispersive boundary layers for scalar conservation laws in one space dimension}

\author[P.~Antonelli]{Paolo Antonelli\orcidlink{0000-0001-8211-8296}}
\address{Gran Sasso Science Institute, viale Francesco Crispi, 7, 67100 L'Aquila, Italy}
\email{paolo.antonelli@gssi.it}

\author[P.~Marcati]{Pierangelo Marcati\orcidlink{0000-0001-6528-6562}}
\address{Gran Sasso Science Institute, viale Francesco Crispi, 7, 67100 L'Aquila, Italy}
\email{pierangelo.marcati@gssi.it}

\author[L.~V.~Spinolo]{Laura V.~Spinolo\orcidlink{0000-0002-7696-8873}}
\address{L.V.S. CNR-IMATI ``E. Magenes'', via Ferrata 5, I-27100 Pavia, Italy.}
\email{spinolo@imati.cnr.it}

 
\begin{document}

\begin{abstract}
We study the zero-dispersion limit for a class of Korteweg-De Vries (KdV)-type initial-boundary value problems on the half-line,
with Dirichlet boundary conditions assigned at \(x=0\). We focus on the outflow regime, where the solution of the limiting scalar conservation law does not attain the boundary condition imposed on the dispersive problem.

We construct a boundary-layer profile, depending on the fast variable, which is uniquely determined, through the
associated stationary third-order boundary-layer equation,
by the mismatch between the boundary conditions, and by the exponential decay at infinity in the fast variable.
Our main result shows that, under suitable
regularity and compatibility assumptions on the data, the dispersive solution is
well approximated by a WKB expansion given by the sum of the smooth solution of the conservation law and the boundary-layer profile. More precisely, we establish the stability of the boundary-layer profile by deriving quantitative estimates for the remainder term, and show that it converges to $0$ in $H^1$ uniformly in time and up to the lifespan of the smooth solution of the conservation law. 

The proof is based on the analysis of a linearized energy functional and does
not rely on complete integrability or inverse scattering techniques.  It applies
to general fluxes and
requires no smallness assumption on the amplitude of the boundary layer.

To the best of our knowledge, this is the first stability result for boundary layers of KdV-type equations on the half-line. 
\end{abstract}
\maketitle

\section{Introduction}
This paper studies the stability of dispersive boundary layers arising in the zero-
dispersion limit of scalar conservation laws posed on the half-line. Boundary layers are 
introduced to resolve the mismatch between the boundary condition imposed on the dispersive 
problem and the trace selected by the limiting conservation law.  
Our main result rigorously justifies the approximation of the dispersive solution by the sum of the hyperbolic solution and the boundary layer, 
with a remainder term that vanishes as the dispersion parameter tends to zero.
Our estimates
hold up to the lifespan of the smooth limiting solution and do not require smallness
assumptions on the data.

We consider the initial-boundary value problem 
\begin{equation}\label{eq:KdV}
\left\{\begin{aligned}
&\d_tu^\eps+\d_x[f(u^\eps)]+\eps^2\d_{xxx}u^\eps=0,\qquad(t, x)\in\R_+\times\R_+\\
&u^\eps|_{t=0}=u_{in},\quad u^{\eps}|_{x=0}=u_b,
\end{aligned}\right.
\end{equation}
where the initial and boundary data satisfy the strong regularity assumptions specified in~\eqref{eq:regdata} and~\eqref{e:nonmipiace} below. We also assume
\begin{equation} \label{eq:c}
   f \in C^4 (\R), \quad  f'(u) \leq - c <0, \quad \text{for every $u \in J$}   \end{equation}
for some constant $c>0$ and open interval $J \subseteq \R$. Throughout the paper, we tacitly assume that the solutions of our equations attain values in $J$ (see also Remark~\ref{r:compact} below). In particular, our analysis applies to the archetypal Korteweg–De Vries (KdV) equation 
$$
\d_tu^\eps+ 6 u^\eps \d_x u^\eps +\eps^2\d_{xxx}u^\eps=0,
$$
provided $f(u)=3u^2$ and we work in regimes where $u$ is negative and bounded away from 0 (see again Remark~\ref{r:compact} below). Note however that the fact that we work with a more general flux function $f$ rules out the possibility of leveraging the rich structure stemming from the complete integrability of the KdV equation.

The present paper is devoted to the analysis of the zero dispersion limit $\eps \to 0^+$ of~\eqref{eq:KdV}. In this regime, and away from the boundary, the initial-boundary value problem~\eqref{eq:KdV} formally reduces to 
\begin{equation}\label{eq:burg}
\left\{
\begin{array}{ll}
\d_tu+\d_x[f(u)]=0,\qquad(t, x)\in\R_+\times\R_+ \\
\left. u \right|_{t=0}= u_{in}. 
\end{array}
\right.
\end{equation}
Note that, owing to~\eqref{eq:c}, the characteristic lines of the scalar conservation law at the first line of~\eqref{eq:burg} have negative slope and hence the values of $u$ at the domain boundary $x=0$ are completely determined by $u_{in}$, see~\cite{BardosLeRouxNedelec}. In particular, in general the function $u$ does \emph{not} satisfy the boundary condition in~\eqref{eq:KdV}, namely $u(0, \cdot) \neq u_b$. An analogous situation occurs in the vanishing viscosity (or diffusion) limit, and in that case the by now classical result by Xin~\cite{Xin} ensures that the transient behavior near the domain boundary 
due to the mismatch between the boundary values attained by the solutions of the viscous and the inviscid conservation law
is accurately described  by the so-called \emph{viscous boundary layers}. 

From both the analytical and the physical perspectives, the study of the zero-dispersion and zero-viscosity limits differs substantially. In a nutshell, the scope of the present paper is to provide a complete boundary-layer analysis in the zero \emph{dispersion} limit. To the best of our knowledge, this is the first result concerning KdV-type equations. 

Towards this end, we impose strong regularity assumptions on $u_{in}$, which imply that~\eqref{eq:burg} has a regular solution defined on some nontrivial time interval $[0, T_0]$ with $T_0>0$. More precisely, 
\begin{equation} \label{eq:regularity}
  u \in C^0 ([0, T_0]; H^4(\R_+)).
\end{equation}
We restrict our analysis to a fixed time interval $[0, T_0]$ where $u$ is regular. Note, in particular, that this rules out the unstable behaviors observed in the vanishing dispersion limit when the entropy admissible solution of the scalar conservation law has discontinuities, see for instance~\cite{GravaKlein}. In order to describe the transient behavior close to the boundary, we introduce the \emph{dispersive boundary-layer profile} $V: [0, T_0] \times \R_+ \to \R$ satisfying  
\begin{equation}\label{eq:bl_eq}
\left\{\begin{aligned}
&\d_y\left[f(u^0+V)\right]+\d_{yyy}V=0\\
&V(t, 0)=u_b(t)-u^0(t) \\
&\lim_{y \to + \infty} V(t, y)= 0,\,\textrm{for every $t$,}
\end{aligned}\right. 
\end{equation}
where 
\begin{equation}\label{e:uzero}
    u^0(t): = u(t, 0).
\end{equation} 
Equation~\eqref{eq:bl_eq} is obtained as the leading order term of the expansion of~\eqref{eq:KdV} with respect to the fast variable $y=x/\eps$. As it turns out, the boundary value problem~\eqref{eq:bl_eq} has a unique solution $V \in C^3 (\R_+ \times \R)$ which decays exponentially fast to $0$ as $y\to + \infty$, see the analysis in \S\ref{s:bl}. This implies, in particular, that, for every fixed $t \in [0, T_0]$, the rescaled function $u^0(t)+ V (t, \cdot/\eps)$ provides a steady solution of the equation at the first line of~\eqref{eq:KdV} and describes the sharp transition, which becomes steeper as $\eps \to 0^+$, between the boundary value attained by $u^\eps$ and by $u$. To complete the problem setup, we introduce the so-called \emph{WKB expansion} 
\begin{equation}\label{eq:exp}
u^\eps(t, x)=u(t, x)+V\left(t, \frac{x}{\eps}\right)+w^\eps(t, x).
\end{equation}
In the jargon of boundary-layer analysis, proving the \emph{boundary-layer stability} amounts to establishing the validity of the above expansion, that is to prove that the remainder term $w^\eps$ vanishes in the $\eps \to 0^+$ limit, in a suitable topology.  

We are now in a position to state our main result, which establishes the well-posedness of the initial-boundary value problem~\eqref{eq:KdV}, and the boundary-layer stability. 
\begin{theorem}
\label{t:main}
Assume~\eqref{eq:c},~\eqref{eq:regularity} and that the initial and boundary data satisfy the following regularity and compatibility conditions:
\begin{equation}\label{eq:regdata}
u_{in} \in H^4 (\R_+), \quad u_b \in H^4 ([0, T_0]) 
\end{equation} 
and
\begin{equation} \label{e:nonmipiace}
     u_b (0) = u_{in}(0), \quad \d_{xxx} u_{in} (0)=0 , \quad  \d_x u_{in}(0)=0, \quad u'_b (0)=0. 
\end{equation}
Then there is $\eps_\ast>0$ such that for every $\eps \in ]0, \eps_\ast]$ the following holds. \\
First, the initial-boundary value problem~\eqref{eq:KdV} has a unique solution $$u^\eps \in C^0([0, T_0]; H^2 (\R_+))~\cap~L^\infty ([0, T_0]; H^3 (\R_+)).$$ Second, we have
\begin{equation} \label{eq:main}
      \int_{\R_+} \left[(w^\eps)^2 + \frac{\eps^2}{2}  (\d_x w^\eps)^2 \right] (t, x)  \ dx  \leq K_\ast \eps^3 , \quad \text{for every $t \in [0, T_0]$,}
\end{equation}
provided $w^\eps$, $u$ and $V$ are given by~\eqref{eq:exp},~\eqref{eq:burg} and~\eqref{eq:bl_eq}, respectively.
In particular, $ w^\eps \to 0$ in $C^0([0, T_0]; H^1(\R_+))$ and henceforth in the uniform norm. The constants $\eps_\ast$ and $K_\ast$ only depend on the following quantities: $T_0, \| u_b - u^0 \|_{C^3 ([0, T_0])}, \| f \|_{C^4}, c$ and  $\| u\|_{C^0(H^4)}$.
\end{theorem}
The most relevant result in the above theorem is the boundary-layer stability, namely estimate~\eqref{eq:main}. Remarkably, we establish stability up to the chosen existence time $T_0$ of the smooth solution of the conservation law~\eqref{eq:burg}, and do not require any smallness assumption on the initial and boundary data. We point out that, in the viscous framework, the stability result in~\cite{Xin} holds: i) for \emph{weak} (that is, small amplitude) boundary layers and on a time interval $[0, \hat T]$ for some $\hat T$ in general smaller than $T_0$; ii) for general boundary layers and up to the existence time $T_0$ of the smooth solution $u$, but under a further convexity assumption on $f$. 

Concerning the existence and uniqueness result for~\eqref{eq:KdV}, we mention in passing that our proof actually establishes uniqueness in the slightly wider regularity class defined by~\eqref{e:regularity}. Note furthermore that the global well-posedness of the initial-boundary value problem for the Korteweg–De Vries (KdV) equation was investigated in previous works, most notably~\cite{BonaWinther,BonaSunZhang,BonaSunZhang2,BonaSunZhang3,Bubnov,CollianderKenig, Faminskii,FokasHimonasM,Holmer}. In particular, the analysis in~\cite{BonaSunZhang} relies on the refined and powerful techniques developed in the context of nonlinear dispersive wave equations. Although it is likely that the results in~\cite{BonaSunZhang} can be extended to deal with more general nonlinearities, we decided to provide here a different existence and uniqueness proof, which relies, for the existence part, on the same parabolic approximation as in~\cite{BonaWinther}. Our approach is more restrictive compared to the one in~\cite{BonaSunZhang}, as we only establish well-posedness up to the lifespan of the smooth solution of~\eqref{eq:burg}. The reason why we decided to provide a detailed proof is that our approach appears much simpler and more elementary than the one that could follow from~\cite{BonaSunZhang}. 

Concerning the references to the existing literature, we refer to the classical volume by Dafermos~\cite{Dafermos} for a comprehensive introduction to conservation laws. The analysis of \emph{viscous} boundary layers has received enormous attention in the last thirty or more years, and many authors have investigated stability (and instability) phenomena when the vanishing viscosity limit is a system of conservation laws in one or several space dimensions. Rather than providing a, necessarily incomplete, list of references, we refer to the books by Serre~\cite{SerreI,SerreII} and to the overview by Grenier, Guo and Nguyen~\cite{GrenierGuoNguyen} for a more comprehensive introduction. By contrast, to the best of our knowledge, far fewer contributions have focused on the analysis of \emph{dispersive} boundary layers; among them,~\cite{ChironRousset,GuiZhang} establish boundary-layer stability results for the nonlinear Schr\"odinger equation. See also~\cite{BreschLannesMetivier,LinaresPilodSaut} for other related problems. Finally, we mention that the analysis of the zero dispersion limit of the KdV equation, defined on the whole real line, is another topic that has received wide attention. Since this topic is only tangentially related to the present work, here we only mention the milestone works~\cite{LaxLevermore}, which heavily rely on complete integrability tools to characterize the limit. 

To conclude the discussion on Theorem~\ref{t:main}, we point out that a remarkable difference with the viscous case is that the analysis in~\cite{Xin} straightforwardly extends to the case where $f' \ge c >0$ and $x \in \R_-$ by applying the change of variables $x \mapsto - x$. This is no longer true in the dispersive case, and the case $f' \ge c >0$ and $x \in \R_-$ is fairly different from the one studied here, for two main reasons. First, if $x \in \R_-$, the initial-boundary value problem for the equation at the first line of~\eqref{eq:KdV} is under-determined if we only assign a Dirichlet boundary condition at $x=0$, see for instance~\cite{MR4025574,Holmer}. To obtain a well-posed problem, we can for instance assign both a Dirichlet \emph{and} a Neumann condition at $x=0$.  Second and more importantly, in the linear case $f'=c$ the only steady solution of the equation at the first line of~\eqref{eq:KdV} that decays exponentially fast to $- \infty$ is the trivial one identically equal to $0$. The analysis of the case $f' \ge c >0$ and $x \in \R_-$ is therefore nontrivial and calls for further investigation. 
\begin{remark}
Our main estimate~\eqref{eq:main} implies that the $H^1$ norm of $w^\varepsilon$ decays to $0$ in the vanishing $\varepsilon$ limit, with order at least $\sqrt{\varepsilon}$. One could obtain the decay in Sobolev spaces of higher regularity, and/or an improved convergence rate, by refining the WKB approximation~\eqref{eq:exp} and introducing a higher-order expansion, as is customary in the literature, see for instance~\cite{GreGu} and the references therein.
\end{remark}

\subsection*{Paper outline} The exposition is organized as follows. In \S \ref{s:bl} we discuss some preliminary results on the boundary layers defined by~\eqref{eq:bl_eq}. In \S\ref{s:main} we establish the main result of the present paper, namely the stability estimate~\eqref{eq:main}. In \S\ref{s:wp} we prove the well-posedness of~\eqref{eq:KdV}. Finally, in the Appendix we collect the proof of some auxiliary results. For the reader's convenience, we conclude the introduction by collecting the main notation used in the paper.

\begin{remark} \label{r:compact}
    In the following, we will assume that  $f \in C^4_b (\R)$, namely that $f$ is $4$-times continuously differentiable and that its derivatives up to the order $4$ are bounded in the uniform norm. The boundedness assumption is actually redundant. Indeed, our analysis shows, in particular, that the solutions $u^\eps$ of~\eqref{eq:KdV} are bounded in $L^\infty$, uniformly in $\eps$. In other words, $u^\eps$ always attains values in some compact set where, owing to the $f \in C^4$ assumptions, $f$ and its derivatives up to the $4$-th order are bounded. 
    To see that $u^\eps$ is uniformly bounded, we combine~\eqref{eq:main} with the Sobolev-Gagliardo-Nirenberg Inequality to get 
    \begin{equation*}
       \| w^\eps (t, \cdot) \|_{L^\infty} \leq  
       \sqrt{2}  \| w^\eps (t, \cdot) \|^{1/2}_{L^2 } \| \d_x w^\eps (t, \cdot) \|^{1/2}_{L^2 } \stackrel{\eqref{eq:main}}{\leq} \sqrt{2 K_\ast} \eps, 
       \quad \text{for every $t \in [0, T_0]$.}
    \end{equation*}
    Next, we establish a uniform $L^\infty$ bound on $V$ (see~\eqref{e:decayV}) and recalling~\eqref{eq:exp} and~\eqref{eq:regularity} we conclude that $u^\eps$ is bounded uniformly in $\eps$ in the $L^\infty$ norm. 
\end{remark}
\subsection*{Notation}
We denote by $C(a_1, \dots, a_n)$  a constant only depending on the quantities $a_1, \dots, a_n$. Its precise value can vary from occurrence to occurrence.  We also set 
\begin{equation} \label{e:C}
    \hat C : = C( T_0, \| u_b - u^0 \|_{C^3}, \| f \|_{C^4}, c, \| u\|_{C^0(H^4)}),
\end{equation}
where we have used Remark~\ref{r:compact}.  
\subsubsection*{Main mathematical symbols}
\begin{itemize}
\item $\R_+: = ]0, + \infty[$;
\item $L^2$ and $L^\infty$: the Lebesgue spaces of square integrable and essentially bounded functions, respectively. We denote by $\| \cdot \|_{L^2}$ and $\| \cdot\|_{L^\infty}$ the corresponding norms. If $w$ is a measurable function depending on two variables $(t, x)$, we denote by  $\| w(t, \cdot)\|_{L^2}$ the $L^2$ norm of the function obtained by keeping $t$ fixed and integrating only with respect to $x$. 
\item $H^k$: the Sobolev space $W^{ k,2}$, $k \in \mathbb N$. 
\item $H^{-1}$: the dual space of $H^1_0$, and by $\| \cdot\|_{H^{-1}}$ its norm.   
\end{itemize}

\subsubsection*{Symbols introduced in the present paper}
\begin{itemize}
\item $u^\eps$: the solution of the dispersive initial-boundary problem, see~\eqref{eq:KdV};
\item $u_{in}, u_b$: the initial and boundary data attained by $u^\eps$, see again~\eqref{eq:KdV};
\item $c$: the bound from above on $f'$, see~\eqref{eq:c};
\item $u$: the solution of the scalar conservation law, see~\eqref{eq:burg};
\item $T_0$: a fixed existence time of a smooth solution of the scalar conservation law, see~\eqref{eq:regularity};
\item $u^0:= \left. u \right|_{x=0}$, {see~\eqref{e:uzero}};
\item $V$: the boundary-layer profile, see~\eqref{eq:bl_eq};
\item $w^\eps$: the remainder term defined by~\eqref{eq:exp};
\item $u^\eps_a$: the approximate solution $u + V$, see~\eqref{eq:u_app}; 
\item $g$: the coefficient in the second order expansion of $f$, see~\eqref{eq:g};
\item $\mathcal E^b, \mathcal E^{inn}$: the error terms defined by~\eqref{eq:ein} and~\eqref{eq:eb}, respectively; 
\item $W^\eps:$ the anti-derivative of $w^\eps$, see~\eqref{eq:W}; 
\item $w^{\eps \nu}$: the approximate solution satisfying~\eqref{eq:w_nu};
\item $Z^{\eps \nu}:$ the anti-derivative of $\d_t w^{\eps \nu}$, see~\eqref{e:Zeta}. 
\end{itemize}

\section{Preliminary results on the boundary-layer equation}
\label{s:bl}
In this section we establish some a priori estimates on the solution of the boundary-layer system~\eqref{eq:bl_eq} that we need in the following.

To study the well-posedness of~\eqref{eq:bl_eq} we fix two constants $\bar u^0$ and $\bar V$ and introduce the boundary value problem 
\begin{equation}\label{eq:bl_eq2}
\left\{\begin{aligned}
&\d_y\left[f(\bar u^0+V)-f(\bar u^0)\right]+\d_{yyy}V=0\\
&V(y=0) = \bar V, \quad \lim_{y \to + \infty} V(y)= 0,
\end{aligned}\right. 
\end{equation}
which is basically obtained from~\eqref{eq:bl_eq} by freezing the $t$ variable. 
\begin{lemma}
\label{l:bl}
Assume~\eqref{eq:c}. For every $\bar u^0, \bar V \in \R$, the boundary value problem~\eqref{eq:bl_eq2} has a unique solution $V \in C^3 \cap H^1 (\R_+)$. Also, 
    \begin{equation}\label{e:decay}
        |V (y)| 
          \leq |\bar V| \exp \left( - \sqrt{c} y\right).
    \end{equation}      
\end{lemma}
With Lemma~\ref{l:bl} in place, we establish the well-posedness of~\eqref{eq:bl_eq} applying, for any given $t \in [0, T_0]$, Lemma~\ref{l:bl} with $\bar V: = u_b(t) - u^0(t)$ and $\bar u^0= u^0 (t)$. Note that~\eqref{e:decay} yields 
\begin{equation} \label{e:decayV}
 |V (t, y)| \leq \| u_b - u^0 \|_{C^0}
 \exp \left( - \sqrt{c} y\right)
\end{equation}
and that, owing to the first condition in~\eqref{e:nonmipiace}, this yields 
\begin{equation}
    \label{e:V000}
    V (0, y) =0, \quad \text{for every $y$.}
\end{equation}
The proof of Lemma~\ref{l:bl} relies on fairly standard arguments, and for completeness we provide it in the Appendix. 

In the following, we need to establish exponential decay properties for the derivatives $\d_t V$, $\d_{tt} V$ and $\d_{ttt} V$. Towards this end, we establish a preliminary elementary lemma. 
\begin{lemma}\label{l:ellittico}
Assume~\eqref{eq:c},\eqref{eq:regularity} and~\eqref{eq:regdata}. Let $V$ be the solution of~\eqref{eq:bl_eq} and $\sigma: \R_+ \to \R$ be a summable function satisfying 
\begin{equation}
    \label{e:sigma}
     |\sigma (y) | \leq D \exp[- d  y] 
\end{equation}
for some $d, D >0$. Assume that $Z: \R_+ \to \R$ is a twice continuously differentiable function satisfying  
\begin{equation}\label{e:ode}
     f'(u^0 + V) Z +  Z'' + \sigma =0, \quad \lim_{y \to + \infty} Z(y)=0.
\end{equation}
Then 
\begin{equation}
\label{e:wanteddecay}
  |Z(y)| \leq  C  (c,  D, Z(0) ) \exp(- 
   \min \{d, \sqrt{c}/2 \} y).  
\end{equation}
\end{lemma}
\begin{proof}
The proof is a simple application of the maximum principle, which we detail for the sake of completeness.
We set $U: = Z- B e^{-by}$, where $B$ and $b$ are suitable constants satisfying 
\begin{equation}
\label{e:B}
    B > \max \left\{ |Z(0)|, \frac{4D}{3c}\right\}, \quad 
    0 < b < \min \left\{ d^2, \sqrt{c}/2  \right\}.
\end{equation}
Note that by combining the above inequalities we get $b^2 < c - D/B$. Using~\eqref{e:ode}, we compute 
\begin{equation} \label{e:mpU}\begin{split}
          & f'(u_0+V) U + U''   = \underbrace{f'(u_0+V) Z + Z''}_{= - \sigma} + [- f' (u_0+V) - b^2]Be^{-by} \\
          & \stackrel{\eqref{e:sigma}}{\ge}
          - D e^{-dy} + [- f' (u_0+V) - b^2 ]Be^{-by}
          \stackrel{b \leq d}{\ge}
          \left[- \frac{D}{B} - f' (u_0+V) - b^2 \right]Be^{-by}
          \\
          & \stackrel{\eqref{eq:c}}{\ge}
          \left[- \frac{D}{B} +c - b^2 \right]Be^{-by} 
          \stackrel{\eqref{e:B}}{>}0
               \end{split}
\end{equation}
Assume by contradiction that $U(y)$ attains strictly positive values for some $y \in \R_+$. Since $U(0)<0$ and  $\lim_{y \to + \infty} U(y)=0$, there is $y^\ast \in \R_+$ such that $U(y^\ast)= \max_{y \in \R_+} U(y)>0$. By using~\eqref{e:mpU} we obtain $U''(y^\ast) > f'(u_0+V) U(y^\ast) >0$, which contradicts the maximality of $U(y^\ast)$. This yields $Z (y)\leq B e^{-b y}$. The proof of the inequality $Z (y)\ge- B e^{-b y}$ is entirely analogous and is therefore omitted. \end{proof}
\begin{lemma} \label{l:decayV0t}
Assume~\eqref{eq:c}, \eqref{eq:regularity} and~\eqref{eq:regdata}. Let $V$ be the solution of~\eqref{eq:bl_eq}, then  
\begin{equation}
    \label{e:decayV0t}
        |\d_t V (t, y)| \leq \hat C 
 \exp \left(  - \frac{\sqrt{c}}{2}\ y\right)
\end{equation}
and 
\begin{equation}
    \label{e:Vt0a0}
     \d_t V(0, y) = 0 \quad \text{for every $y$}.
\end{equation}
We also have 
\begin{equation} \label{e:decayV0tt}
     |\d_{tt} V (t, y)| \leq \hat C 
 \exp \left( - \frac{\sqrt{c}}{2}y\ \right) 
\end{equation}
and 
\begin{equation} \label{e:decayV0ttt}
     |\d_{ttt} V (t, y)| \leq \hat C 
 \exp \left(  - \frac{\sqrt{c}}{2}\ y\right), 
\end{equation}
\end{lemma}
\begin{proof}
We first establish~\eqref{e:decayV0t}. First, we point out that classical results on the continuous dependence of the solutions of ODEs on parameters imply that $\d_t V$ is a continuously differentiable function. Also, $\lim_{y \to + \infty} \d_t V(t, y) =0$ owing to~\eqref{e:decay}. To gather further information on $\d_t V$, we consider~\eqref{eq:bl_eq} and compute the $t$-derivative. We arrive at 
\begin{equation} \label{e:eqV0t}
   [f'(u^0 + V) - f'(u^0)] [u^0]' + f'(u^0 + V) \d_t V + \d_{yy t} V=0
\end{equation}
and observe that 
$$
  |f'(u^0 + V) - f'(u^0)| \leq \hat C |V| \stackrel{\eqref{e:decayV}}{\leq}
  \hat C \| u_b - u^0 \|_{C^0}
 \exp \left( - \sqrt{c} y\right).
$$
We apply Lemma~\ref{l:ellittico} with $\sigma =  [f'(u^0 + V) - f'(u^0)] [u^0]' $, $d=\sqrt{c}$ and $Z= \d_t V$ and arrive at~\eqref{e:decayV0t}. To establish ~\eqref{e:Vt0a0}, we plug the second condition  in~\eqref{e:nonmipiace} into the equation at the first line of~\eqref{eq:burg} to get $[u^0]'(0)= [f(u_{in})]'(0)=0$, which combined with the last condition in~\eqref{e:nonmipiace} yields $[u^0 - u_b]'(0)=0$. This implies that $\d_t V(0, 0) =0$. Owing to~\eqref{e:V000}, the source term in~\eqref{e:eqV0t} vanishes at $t=0$. By the uniqueness of solutions of the elliptic equation in $H^1_0(\R_+)$, this yields~\eqref{e:Vt0a0}. 

To establish~\eqref{e:decayV0tt}, we compute the $t$-derivative of~\eqref{e:eqV0t} and arrive at 
\begin{equation} \label{e:eqV0tt}
\begin{split}
   [f''& (u^0 + V) - f''(u^0)] \big[[u^0]'\big]^2
   + f''(u^0 + V) \d_t V  +
   [f'(u^0 + V) - f'(u^0)] [u^0]'' \\ &
   + f''(u^0 + V) [u^0]' \d_t V  
     + f''(u^0 + V) [ \d_t V ]^2   + 
      f'(u^0 + V)  \d_{tt} V  + \d_{yy tt} V=0
    \end{split}
\end{equation}
We then apply Lemma~\ref{l:ellittico} with $Z= \d_{tt} V$, $d=\sqrt{c}/2$ and $\sigma$ given by the sum of all but the last two terms in~\eqref{e:eqV0tt}. Note that~\eqref{e:sigma} comes from~\eqref{e:decay} and~\eqref{e:wanteddecay}. To establish~\eqref{e:decayV0ttt}, we evaluate the $t$-derivative of~\eqref{e:eqV0tt} and arrive at a cumbersome expression, which however can be written in the form~\eqref{e:ode} provided $Z= \d_{ttt} V$ and $\sigma$ is the sum of a bunch of terms that satisfy~\eqref{e:sigma} with $d=\sqrt{c}/2$ owing to~\eqref{e:decay},~\eqref{e:decayV0t} and~\eqref{e:decayV0tt}.   
\end{proof}

\section{Proof of the boundary-layer stability}
\label{s:main}
In this section, we assume the well-posedness of  
~\eqref{eq:KdV}, which is proved in the next section, and we analyze the stability of the boundary layers. Namely, we prove that the remainder term $w^\eps$ defined as in~\eqref{eq:exp} satisfies~\eqref{eq:main}. The exposition is organized as follows: in \S\ref{s:weps} we derive the equation satisfied by $w^\eps$, whereas in \S\ref{s:proof} we provide the detailed proof of~\eqref{eq:main}.  
\subsection{The equation satisfied by the remainder term}\label{s:weps} 
For convenience, we define the approximate solution $u^\eps_a$ by setting 
\begin{equation}\label{eq:u_app}
u^\eps_a(t, x)=u(t, x)+V\left(t, \frac{x}{\eps}\right).
\end{equation}
We now plug~\eqref{eq:exp} into the equation at the first line of~\eqref{eq:KdV}, and recall that $u$ and $V$ satisfy~\eqref{eq:burg} and ~\eqref{eq:bl_eq}, respectively.
We obtain
\begin{equation*}\begin{aligned}
0=&\d_t u + \d_tV+\d_tw^\eps+\d_x [f(u^\eps)]+\eps^2\d_{xxx}u+\frac1\eps\d_{yyy}V+\eps^2\d_{xxx}w^\eps\\
\stackrel{\eqref{eq:burg},\eqref{eq:bl_eq}}{=}&\d_tV+\d_tw^\eps+\d_x\left[f(u^\eps)-f(u)-f(u^0+V)+f(u^0)\right]\\
&+\eps^2\d_{xxx}u
+\eps^2\d_{xxx}w^\eps.
\end{aligned}\end{equation*}
We now Taylor-expand 
\begin{equation*}
f(u^\eps)=f(u^\eps_a)+f'(u^\eps_a)w^\eps+g(u^\eps_a, w^\eps)(w^\eps)^2,
\end{equation*}
where
\begin{equation} \label{eq:g}
    g(u^\eps_a, w^\eps): = \int_0^1 \int_0^\xi f''(u^\eps_a+ \eta w^\eps) \, d \eta \ d \xi, 
\end{equation}
and conclude that $w^\eps$ is a solution of the initial-boundary value problem
\begin{equation}\label{eq:w_eps}
\left\{\begin{aligned}
&\d_tw^\eps+\d_x[f'(u^\eps_a)w^\eps+g(u^\eps_a, w^\eps)(w^\eps)^2]+\eps^2\d_{xxx}w^\eps+\mathcal E^{inn}+\mathcal E^{b}=0,\\
&w^\eps|_{t=0}=0=w^\eps|_{x=0}, 
\end{aligned}\right.
\end{equation}
provided 
\begin{equation} \label{eq:ein}
\mathcal E^{inn}: =\eps^2\d_{xxx}u,
\end{equation}
\begin{equation} \label{eq:eb}
\mathcal E^b : =\d_tV+\d_x[f(u^\eps_a)-f(u)]-\d_x[f(u^0+V)-f(u^0)].
\end{equation}
\subsection{Proof of~\eqref{eq:main}} \label{s:proof}
The basic idea underpinning the proof of Theorem~\ref{t:main} is that control on the left-hand side of~\eqref{eq:main}
comes from the analysis of 
\begin{equation} \label{eq:E}
        E (t) : = \int_{\R_+}\frac{\eps^2}{2}(\d_xw^\eps)^2-f'(u^\eps_a)(w^\eps)^2\,dx ,
\end{equation}
which, owing to~\eqref{eq:c}, may be interpreted as a linearized energy around $u^\eps_a$. In this respect, our analysis is indebted to the one in previous works that highlighted the role of similar functionals in obtaining a priori estimates for KdV-type equations, see for instance~\cite{Ginibre,GinibreTsutsumiVelo}. 

To control $E$, we introduce the anti-derivative of $w^\eps$ by setting 
\begin{equation} \label{eq:W}
W^\eps(t, x): =\int_0^xw^\eps(t, x')\,dx' \implies  \partial_x W^\eps = w^\eps, 
\end{equation}
and point out that
\begin{equation}\label{eq:dataz}
     \left. W^\eps \right|_{t=0}=0, \qquad \left. W^\eps \right|_{x=0}=0,
\end{equation}
where the first condition comes from the initial condition in~\eqref{eq:w_eps}. 
Next, we fix a time $T>0$, to be specified in the following, and multiply the equation at the first line of \eqref{eq:w_eps} by $\d_tW^\eps$. We integrate in space and time over the interval $[0, T] \times \R_+$ and arrive at 
\begin{equation}\begin{split} \label{eq:I1I6}
0=& \int_0^T \! \int_{\R_+}\d_tW^\eps\Big[\d_tw^\eps+\d_x[f'(u^\eps_a)w^\eps+g(u^\eps_a, w^\eps)(w^\eps)^2]+\eps^2\d_{xxx}w^\eps  \\
 &  + \mathcal E^{inn}+\mathcal E^{b}\Big]\,dx dt =: I_1+\dotsc + I_6,
\end{split}\end{equation}
where $I_1, \dots, I_6$ are defined below. 
The terms $I_2$ and $I_4$ are the main ones, and correspond (up to higher-order terms) to the time derivative of the linearized energy $E$ defined in~\eqref{eq:E}. The rest of the proof containing all the technical details is organized into the following steps. \\
{\bf Step 1: notation setup and preliminary estimates.} We rely on a continuous induction argument. 
Towards this end, we define $T^\ast$ by setting 
\begin{equation} \label{eq:Tast}
      T^\ast : = \sup \Big\{t \in [0, T_0]: \int_{\R_+} \Big[(w^\eps)^2 + \frac{\eps^2}{2}  (\d_x w^\eps)^2 \Big](\tau, x)  \ dx  \leq \eps^2 \; \text{for every $\tau \in [0, t]$}\Big\}.
\end{equation}
Note that the set for which $T^\ast$ is the supremum is non-empty due to the initial condition $\left. w^\eps \right|_{t=0}=0$ {and the continuity of the map $t \mapsto \| w^\eps(t, \cdot) \|_{H^1}$}. Also, in principle the value of $T^\ast$ might depend on $\eps$. As a matter of fact, we will show that $T^\ast= T_0$ and hence, in particular, its value does not depend on $\eps$. 

The definition of $T^\ast$ combined with the Sobolev-Gagliardo-Nirenberg Inequality yields 
\begin{equation}\label{eq:linftyw}
    \| w^\eps (t, \cdot) \|_{L^\infty_x} \leq  \sqrt{2} \| w^\eps (t, \cdot) \|^{1/2}_{L^2_x} \| \partial_x w^\eps (t, \cdot) \|^{1/2}_{L^2_x} 
    \leq \sqrt{2} \eps^{1/2}   \quad \text{for every $t\in [0, T^\ast]$}. 
\end{equation}
Note furthermore that the H\"older Inequality implies 
\begin{equation} \label{eq:linftyW}
      |W^\eps (t, x)| \leq \sqrt{x} \| w^\eps (t, \cdot) \|_{L^2_x}, 
      \end{equation}
provided $W^\eps$ is the same as in~\eqref{eq:W}. We now fix $T \in [0, T^\ast]$ and control the terms $I_1, \dots, I_6$ satisfying~\eqref{eq:I1I6}. \\
{\bf Step 2: control on $I_1$ and $I_2$}. 
We have 
\begin{equation*}
\begin{split}
      I_1 & : = \int_0^T \! \int_{\R_+}\d_tW^\eps \d_tw^\eps dx dt \stackrel{\eqref{eq:W}}{=}\frac12
       \int_0^T \!  \int_{\R_+}\d_x[(\d_tW^\eps)^2]\,dx dt \\ & =
      -\frac12 \int_0^T \! (\d_tW^\eps(t, 0))^2 dt  \stackrel{\eqref{eq:dataz}}{=}0,
\end{split}
\end{equation*}
and, integrating by parts and using the boundary condition~\eqref{eq:dataz}, 
\begin{equation*}\begin{aligned}
I_2& : = \int_0^T \!\int_{\R_+}\d_tW^\eps\d_x[f'(u^\eps_a)w^\eps]\,dx dt  \stackrel{\eqref{eq:W},\eqref{eq:dataz}}{=}-\int_0^T \! \int_{\R_+}\d_tw^\eps f'(u^\eps_a)w^\eps\,dx\ dt \\
&=-\frac12 \int_0^T \! \int_{\R_+} f'(u^\eps_a)\partial_t [(w^\eps)^2]\,dx dt \\
& =\int_0^T \! \frac{d}{dt}\left[-\frac12\int_{\R_+}f'(u^\eps_a)(w^\eps)^2\,dx\right] dt +
\underbrace{\frac12\int_0^T \! \int_{\R_+}(w^\eps)^2f''(u^\eps_a)(\d_tu+\d_tV)\,dx dt }_{:=I_{22}} \\
& = \frac12\int_{\R_+}-f'(u^\eps_a)(w^\eps)^2 (T, x) \,dx  + I_{22}.
\end{aligned}\end{equation*}
Note furthermore that 
\begin{equation*} \begin{split}
    |I_{22}|  & \leq C ( {\color{black}\| f''\|_{C^0}}, \| \partial_t u \|_{L^\infty}, \| \partial_t V \|_{L^\infty} )\int_0^T \! \! \int_{\R_+} (w^\eps)^2 (t, \cdot) \ dx dt
    \\ & 
    \stackrel{\eqref{eq:regularity},\eqref{e:C},\eqref{e:decayV0t}}{\leq} \hat C
    \int_0^T \! \! \int_{\R_+} (w^\eps)^2 (t, \cdot) \ dx dt.
    \end{split}
\end{equation*}
{\bf Step 3: control on $I_3$}. We have 
\begin{equation*}
\begin{split}
       I_3:  = & \int_0^T \! \int_{\R_+}\d_tW^\eps \d_x [ g(u^\eps_a, w^\eps)(w^\eps)^2]dx \stackrel{\eqref{eq:W},\eqref{eq:dataz}}{=}
      - \int_0^T \! \int_{\R_+}\d_t w^\eps g(u^\eps_a, w^\eps)(w^\eps)^2 dx dt \\ &
      = - \int_0^T \! \frac{d}{dt} \int_{\R_+} G(u^\eps_a, w^\eps) dx dt  + 
       \underbrace{\int_0^T \! \int_{\R_+} \d_1 G(u^\eps_a, w^\eps) \d_t u^\eps_a dx dt}_{:= I_{32}} \\ 
       & \stackrel{w^\eps (0, \cdot)=0}{=} \underbrace{- \int_{\R_+} G(u^\eps_a, w^\eps) (T, x) dx}_{:= I_{31}}  + I_{32}, 
\end{split}
\end{equation*}
provided 
\begin{equation} \label{eq:G}
      G(u^\eps_a, w^\eps) : = \int_0^{w^\eps} g(u^\eps_a, \chi)\chi^2 d \chi \stackrel{\eqref{eq:g}}{=}
     \int_0^{w^\eps} \! \! \chi^2  \int_0^1 \int_0^\xi f''(u^\eps_a+ \eta \chi) \, d \eta \ d \xi \ d \chi,
\end{equation}
the function $g$ is the same as in~\eqref{eq:g}, and $\d_1 G$ denotes the partial derivative of $G$ with respect to the variable $u^\eps_a$,
that is 
\begin{equation*}
    \d_1 G(u^\eps_a, w^\eps)  \stackrel{\eqref{eq:G}}{=}
     \int_0^{w^\eps} \! \!  \chi^2 \int_0^1 \int_0^\xi f'''(u^\eps_a+ \eta \chi)  \, d \eta \ d \xi \ d \chi.
\end{equation*}
Note that 
\begin{equation*}
      |G(u^\eps_a, w^\eps)| \leq \hat C  (w^\eps)^3, \quad 
         |\d_1 G(u^\eps_a, w^\eps)| \leq \hat C (w^\eps)^3,
       \end{equation*}
which yields 
\begin{equation}
    |I_{31}| \leq \hat C \int_{\R_+} (w^\eps)^3 (T, x) dx \stackrel{\eqref{eq:linftyw}}{\leq}
   \hat C \eps^{1/2} \int_{\R_+} (w^\eps)^2 (T, x) dx
\end{equation}
and 
\begin{equation}
    \begin{split}
    |I_{32}| & \leq  \hat C  \| \partial_t u^\eps_a  \|_{L^\infty} 
    \int_0^T \! \! \int_{\R_+} |w^\eps|^3  \ dx dt  \stackrel{\eqref{eq:linftyw}}{\leq} \hat C  \eps^{1/2 } \| \partial_t u^\eps_a  \|_{L^\infty} 
    \int_0^T \! \! \int_{\R_+} [w^\eps]^2  \ dx dt \\
    & 
    \stackrel{\eqref{eq:regularity},\eqref{e:decayV0t}, \eqref{eq:u_app}}{\leq} 
    \hat C  \eps^{1/2 } 
    \int_0^T \! \! \int_{\R_+} [w^\eps]^2   \ dx dt.
    \end{split}
\end{equation}
{\bf Step 4: control on $I_4$ and $I_5$}. We have 
\begin{equation*}
\begin{split}
I_4 & : =\eps^2\int_0^T \int_{\R_+}\d_tW^\eps\d_{xxx}w^\eps\,dx dt 
\stackrel{\eqref{eq:W},\eqref{eq:dataz}}{=} - \eps^2 \int_0^T \int_{\R_+}\d_{t}w^\eps\d_{xx}w^\eps\,dx dt \\ &
  \stackrel{\d_t w^\eps(\cdot, 0)\equiv 0 }{=}
   \eps^2 \int_0^T \int_{\R_+}\d_{tx }w^\eps\d_{x}w^\eps\,dx dt  = 
  \frac{\eps^2}{2} \int_0^T \frac{d}{dt}\int_{\R_+}(\d_xw^\eps)^2\,dx dt  \\
  & \stackrel{w^\eps(0, \cdot) \equiv 0 }{=}  \frac{\eps^2}{2} \int_{\R_+}(\d_xw^\eps)^2 (T, x)\,dx
\end{split}
\end{equation*}
and
\begin{equation*} \begin{split}
        I_5 : = & \int_0^T \int_{\R_+} \d_t W^\eps \mathcal E^{inn} \, dx dt  \stackrel{\eqref{eq:ein}}{=} 
      \eps^2 \int_0^T \int_{\R_+} \d_t W^\eps \d_{xxx} u dx dt  \\
  &     \stackrel{\eqref{eq:W},\eqref{eq:dataz}}{=}   -\eps^2 \int_0^T \int_{\R_+} \d_t w^\eps \d_{xx} u \ dx dt \\
     & = - \eps^2 \int_0^T \frac{d}{dt}
       \int_{\R_+} w^\eps \d_{xx} u \ dx dt +   \underbrace{  \eps^2 \int_0^T \int_{\R_+} w^\eps \d_{t xx} u \, dx dt }_{:= I_{52}}\\
       & \stackrel{w^\eps(0, \cdot) \equiv 0}{=} \underbrace{- \eps^2   \int_{\R_+} w^\eps \d_{xx} u (T, x) \ dx}_{: = I_{51}} +  I_{52}. 
\end{split}
\end{equation*}
Note that, owing to the Young Inequality, we have 
\begin{equation*}
     |I_{51}|  \leq \eps \int_{\R_+} (w^\eps)^2(T, x)  dx  + \hat C \eps^3, \quad 
      |I_{52}|  \leq \eps \int_0^T \int_{\R_+} (w^\eps)^2  dx dt + \hat C \eps^3. 
\end{equation*}
{\bf Step 5: control on $I_6$}.
We have 
\begin{equation*} \begin{split} 
        I_6 : = & \int_0^T \int_{\R_+} \d_t W^\eps \mathcal E^{b} dx dt \stackrel{\eqref{eq:eb}}{=} 
       \underbrace{\int_0^T \int_{\R_+} \d_t W^\eps \d_t V dx dt }_{:= I_{61}}  \\
       & \quad + \underbrace{\int_0^T \int_{\R_+} \d_t W^\eps \d_x[f(u^\eps_a)-f(u)- f(u^0+V)+f(u^0)] dx dt}_{:=I_{62}},
\end{split}
\end{equation*}
which implies 
\begin{equation*} 
\begin{split}
     I_{61}  & = \int_0^T \frac{d}{dt} \int_{\R_+} W^\eps \d_t V \, dx  dt \underbrace{ - \int_0^T \int_{\R_+}  W^\eps \d_{tt} V dx \, dt }_{:=I_{612}} \\
     & \stackrel{\eqref{eq:dataz}}{=}   \underbrace{\int_{\R_+} \left. W^\eps \d_t V \right|_{t=T} dx}_{:=I_{611}} + I_{612}. 
     \end{split}
\end{equation*}
By using, among other things, the Young Inequality, we arrive at 
\begin{equation*}
\begin{split}
       |I_{611}|  & \stackrel{\eqref{eq:linftyW}}{\leq}  \| w^\eps (T, \cdot) \|_{L^2}  \int_{\R_+} \sqrt{x} 
      \left| \d_t V \left( t, \frac{x}{\eps}  \right) \right| \ dx  \\ & 
      \stackrel{\eqref{e:decayV0t}}{\leq} \hat C \eps^{1/2}  \| w^\eps (T, \cdot) \|_{L^2}  \int_{\R_+} \sqrt{\frac{x}{\eps}} \exp \left( - \frac{x\sqrt{c}}{2\eps}   \right) \ dx  \\ & 
       \leq \eps^{3/2} \hat C \| w^\eps (T, \cdot) \|_{L^2}   
       \stackrel{\text{Young}}{\leq} \frac{c}{8}
       \| w^\eps (T, \cdot) \|^2_{L^2_x}  + \hat C \eps^{3}.    
\end{split}
\end{equation*}
Similarly,
\begin{equation*}
\begin{split}
       |I_{612}|  & \stackrel{\eqref{eq:linftyW}}{\leq} \int_0^T \! \int_{\R_+} \sqrt{x} \| w^\eps (t, \cdot) \|_{L^2} 
      \left| \d_{tt} V \left( t, \frac{x}{\eps}  \right) \right| \ dx dt \\ & 
      \stackrel{\eqref{e:decayV0tt}}{\leq}  \sqrt{\eps} \int_0^T \| w^\eps (t, \cdot) \|_{L^2}  \int_{\R_+} \sqrt{\frac{x }{\eps}} \exp \left( - \frac{x \sqrt{c} }{ 4 \eps}    \right) \ dx dt \\ & 
       \leq \hat C \eps^{3/2}  \int_0^T \| w^\eps (t, \cdot) \|_{L^2_x} dt \\& 
       \stackrel{\text{Young}, \, T \leq T_0}{\leq}
       \int_0^T  \| w^\eps (t, \cdot) \|^2_{L^2_x} dt +  \hat C \eps^{3}.  
\end{split}
\end{equation*}
To control the term $I_{62}$,  we set 
\begin{equation}\label{eq:acca}\begin{split} 
       H(u, V, u^0): & = f(u^\eps_a)-f(u)- f(u^0+V)+f(u^0) \\ & \stackrel{\eqref{eq:u_app}}{=} V \int_0^1 [f'(u+ \xi V) - f'(u^0+ \xi V) ] d \xi \\
       & = V [u - u^0] \int_0^1 \int_0^1 f''(u^0 + \eta [u- u^0]+ \xi V) \, d \eta \ d \xi.
\end{split}
\end{equation}
Next, we recall that $u^0 = u (x=0)$ and deduce by the Lagrange Theorem that 
\begin{equation}\label{e:Lagrange}
  |u- u^0| (t, x) \leq \| \d_x u (t, \cdot)\|_{L^\infty} x.
  \end{equation}
This in turn implies 
\begin{equation*}
    |H(u, V, u^0)(t, x)| \leq \hat C \left|V\left(t,\frac{x}{\eps}\right)\right| x
\end{equation*}
and hence, owing to~\eqref{e:decay},
\begin{equation}
    \label{eq:stimeH}
       \| H (t, \cdot) \|^2_{L^2} \leq  \hat C \eps^{3}, 
       \quad \text{for every $t \in [0, T]$}.
\end{equation} 
Going back to $I_{62}$, we have 
\begin{equation*}
    \begin{split}
         I_{62} &  \stackrel{\eqref{eq:acca}}{=}
        \int_0^T \! \int_{\R_+} \d_t W^\eps \d_x H dx dt 
        \stackrel{\eqref{eq:W}}{=} - \int_0^T \! \int_{\R_+} \d_t w^\eps H dx dt \\
        & = - \int_0^T  \frac{d}{dt} \int_{\R_+} w^\eps H dx dt + 
          \int_0^T  \! \int_{\R_+} w^\eps \d_t H dx dt \\ &
          \stackrel{w^\eps(0, \cdot) \equiv 0 }{=}
       \underbrace{- \int_{\R_+} \left. w^\eps H \right|_{t=T}  dx}_{: = I_{621}} + 
          \underbrace{\int_0^T  \! \int_{\R_+} w^\eps \d_t H dx dt}_{: = I_{622}}.   
    \end{split}
\end{equation*}
By applying the H\"older and Young Inequalities,  we get 
\begin{equation*}\begin{split}
    |I_{621}| & \stackrel{\text{H\"older}}{\leq} \| w^\eps (T, \cdot) \|_{L^2} \|  H (T, \cdot) \|_{L^2}  
    \stackrel{\text{Young}}{\leq} \frac{c}{8} \|  w^\eps (T, \cdot) \|^2_{L^2_x} +  
    \frac{2}{c} \|  H(T, \cdot) \|^2_{L^2_x}\\
     &
    \stackrel{\eqref{eq:stimeH}}{\leq}
    \frac{c}{8} \|  w^\eps (T, \cdot) \|^2_{L^2_x}  + 
     \hat C \eps^{3}. 
\end{split}
\end{equation*}
To control $I_{622}$, we combine the explicit expression~\eqref{eq:acca} of $H$, the inequality~\eqref{e:Lagrange} and the analogous inequality for $[\d_t u - (u^0)']$ to infer  
\begin{equation*}
\begin{split}
         |\partial_t H (t, x)| & \leq  \hat C \left| \partial_t V \left(t,\frac{x}{\eps}\right) \right| x + 
      \hat C \left|  V \left(t,\frac{x}{\eps}\right) \right| x  \\ & \quad +  \hat C \left|  V \left(t,\frac{x}{\eps}\right) \right| x  \big[ \| \d_t u \|_{L^\infty} + \| \d_t V \|_{L^\infty} \big] \\ & 
      \stackrel{\eqref{eq:regularity},\eqref{e:decayV0t}}{\leq}
      \hat C \left| \partial_t V \left(t,\frac{x}{\eps}\right) \right| x + 
      \hat C \left|  V \left(t,\frac{x}{\eps}\right) \right| x, 
\end{split}
\end{equation*}
which in turn implies, owing to~\eqref{e:decay} and~\eqref{e:decayV0t},  
\begin{equation*}
    \| \partial_t H (t, \cdot) \|^2_{L^2} \leq 
   \hat C \eps^3, \quad \text{for every $t \in [0, T]$.}
\end{equation*}
 Combining the above inequality with the Young and H\"older Inequalities, we eventually arrive at  
\begin{equation*}
    \begin{split}
        |I_{622}| & \leq \int_0^T \! \int_{\R_+} (w^\eps)^2 dx dt + \frac{1}{4} \int_0^T \! \int_{\R_+} (\d_t H)^2 dx dt 
         \leq \int_0^T \! \int_{\R_+} (w^\eps)^2 dx dt +  \hat C \eps^3.
    \end{split}
\end{equation*}
{\bf Step 6: conclusion}.
Combining the estimates at the previous steps, and assuming $\eps \leq 1$, we arrive at 
\begin{equation} \label{eq:wrapup}
    \begin{split}
         \int_{\R_+}& \left[ -\frac12 f'(u^\eps_a) - \frac{c}{4} - C ({\color{black}\| f''\|_{C^0}}) \sqrt{\eps}    
         \right] (w^\eps)^2 (T, x) \,dx  + 
         \frac{\eps^2}{2} \int_{\R_+} (\d_x w^\eps)^2 (T, x) \,dx \\ &
         \leq \hat C \int_0^T \! \! \int_{\R_+} (w^\eps)^2 (t, \cdot) \ dx dt +  \hat C \eps^{3}.   
    \end{split}
\end{equation}
If $\eps$ is small enough to have $C ({\color{black}\| f''\|_{C^0}} )\sqrt{\eps}\leq c/8$, then recalling~\eqref{eq:c} and using the Gr\"onwall Lemma we deduce from the previous inequality  that 
\begin{equation}
\begin{split}
      \| w^\eps (T, \cdot) \|^2_{L^2_x} \leq \hat C  \exp [\hat C T] \eps^{3}. 
\end{split}
\end{equation}
Using the arbitrariness of $T\leq T^\ast$ and plugging the above estimate into~\eqref{eq:wrapup} we conclude that 
\begin{equation} \label{eq:final}
    \begin{split}
         &   \int_{\R_+} \left[(w^\eps)^2 + \frac{\eps^2}{2}  (\d_x w^\eps)^2 \right](T, x)  \ dx \leq \hat C  \exp [\hat C T] \eps^{3}. 
    \end{split}
\end{equation}
We can now conclude our continuous induction argument. We recall the definition~\eqref{eq:Tast} of $T^\ast$, assume by contradiction that $T^\ast < T_0$ and use the continuity of the map $t \mapsto \| w^\eps(t, \cdot) \|_{H^1}$, which implies 
$$
  \int_{\R_+} \left[(w^\eps)^2 + \frac{\eps^2}{2}  (\d_x w^\eps)^2 \right](T^\ast, x) dx
  = \eps^2.
$$
Using~\eqref{eq:final}, we arrive at  
\begin{equation*}
    \begin{split}
           \eps^2 = \int_{\R_+} \left[(w^\eps)^2 + \frac{\eps^2}{2}  (\d_x w^\eps)^2 \right](T^\ast, x)  \ dx \leq \hat C  \exp [\hat C T^\ast] \eps^{3} \stackrel{T^\ast \leq T_0}{\leq} 
            \hat C \eps^3,
     \end{split}
\end{equation*}
which yields a contradiction provided $\eps$ is sufficiently small. This implies that $T^\ast=T_0$, and yields~\eqref{eq:main}. 

\section{Proof of the well-posedness of~\eqref{eq:KdV}}\label{s:wp}
In this section we establish the well-posedness of~\eqref{eq:KdV}. The exposition is organized as follows. In \S\ref{s:uni} we establish uniqueness of the solution belonging to a suitable regularity class. To establish existence, it suffices to prove existence of a solution of the initial-boundary value problem for the remainder term $w^\eps$, namely~\eqref{eq:w_eps}. We introduce the same approximation as in~\cite{BonaWinther}, i.e. 
\begin{equation}\label{eq:w_nu}
\d_t w^{\eps \nu}+\d_x[f'(u^\eps_a)w^{\eps \nu}+g(u^\eps_a, w^{\eps \nu})(w^{\eps \nu})^2]+\eps^2\d_{xxx}w^{\eps \nu} - \nu \d_{txx} w^{\eps \nu} +\mathcal E^{inn}+\mathcal E^{b}=0. 
\end{equation}
In \S\ref{s:app} we establish existence of the corresponding initial-boundary value problem, whereas in \S\ref{s:reg} we prove some uniform-in-$\nu$ regularity estimates, and in \S\ref{s:limit} we pass to the vanishing $\nu$ limit, completing the existence proof. Note that if the analysis in \S\ref{s:app} closely follows the one in~\cite{BonaWinther}, the proof of the regularity estimate in \S\ref{s:reg} is completely different (and more direct) from the one in~\cite{BonaWinther}. 
\subsection{Uniqueness}\label{s:uni}
We show that, if $u_1^\eps$ and $u_2^\eps$ are two solutions of the initial-boundary value problem~\eqref{eq:KdV} both belonging to the regularity class 
\begin{equation}\label{e:regularity}
 u \in L^2 (]0, T_0[; H^2 (\R_+)), \; \d_t u \in L^2 (]0, T_0[, H^{-1}(\R_+)),
\end{equation}
then $u^\eps_1 \equiv u^\eps_2$. Note that the difference $u_1^\eps-u_2^\eps$ satisfies  
 \begin{equation*}
 \left\{
 \begin{array}{ll}
  \d_t [u_1^\eps-u_2^\eps] + \d_x [ f(u_1^\eps)- f(u_2^\eps)] +\eps^2 \d_{xxx}[u_1^\eps-u_2^\eps] =0, \\
  \phantom{cc} \\
  {[  u_1^\eps-u_2^\eps](0, \cdot) =0, \quad [u_1^\eps-u_2^\eps] (\cdot, 0)=0 .} \\
 \end{array}
 \right.
 \end{equation*}
We now fix $T\in ]0, T_0[$, multiply the above equation times $[u_1^\eps - u_2^\eps]$ and integrate in $(t, x)$ over the set $[0, T]\times \R_+$. Owing to~\eqref{e:regularity},  we obtain\footnote{Note that~\eqref{e:regularity} implies, in particular, $u \in C^0([0, T_0]; L^2 (\R_+))$, so the values of $\| [u_1 - u_2](t, \cdot)\|_{L^2}$ are well-defined for every $t \in [0, T_0]$}     
\begin{equation} \label{e:uniqueness}\begin{split}
 \frac{1}{2} & \int_{\R_+} [u_1^\eps - u_2^\eps]^2(T, x) dx + S_1 + S_2 =0,
 \end{split}
\end{equation}
where the terms $S_1$ and $S_2$ are defined as follows. 
We have 
\begin{equation}\label{e:esse1} \begin{split}
 S_1 &: = \int_0^T \! \!  \int_{\R_+} [u_1^\eps - u_2^\eps]  \d_x [ f(u_1^\eps)- f(u_2^\eps)] dx dt \\ & \stackrel{[  u_1^\eps-u_2^\eps](0, \cdot) =0}{=}
 - \int_0^T \int_{\R_+} \d_x [u_1^\eps - u_2^\eps]  [ f(u_1^\eps)- f(u_2^\eps)] dx dt,
 \end{split}
\end{equation}
where the second equality follows from the integration by parts formula. Next, we use the Taylor formula with Lagrange remainder to get 
$$
  f(u_1^\eps)- f(u_2^\eps) = f' (u_1^\eps) [u_1^\eps - u_2^\eps]  + \ell (u_1^\eps, u_2^\eps)[u_1^\eps - u_2^\eps]^2
$$
for some suitable function $\ell$. Plugging the above equality into~\eqref{e:esse1} and then integrating by parts, we get 
\begin{equation*}\begin{split}
 S_1 &: = - \frac{1}{2}\int_0^T \! \!  \int_{\R_+} \d_x\big[ [u_1^\eps - u_2^\eps]^2 \big]  f'(u_1^\eps) dx dt 
 - \frac{1}{3}\int_0^T \! \!  \int_{\R_+} \d_x\big[ [u_1^\eps - u_2^\eps]^3 \big] \ell (u_1^\eps, u_2^\eps) dx dt \\
 & = \frac{1}{2}\int_0^T \! \!  \int_{\R_+}  [u_1^\eps - u_2^\eps]^2 \d_x  [f'(u_1^\eps)]  dx dt + \frac{1}{3}\int_0^T \! \!  \int_{\R_+}  [u_1^\eps - u_2^\eps]^3 \d_x[ \ell (u_1^\eps, u_2^\eps) ] dx dt,
 \end{split}
\end{equation*}
whence 
\begin{equation}\label{e:stimaesse1}
   |S_1| \leq C (\| f \|_{C^3})\int_0^T \Big[ \| \d_x u^\eps_1 \|_{L^\infty}+ \| \d_x u^\eps_2 \|_{L^\infty} \Big] 
   \| [u_1^\eps - u_2^\eps](t, \cdot) \|^2_{L^2}dt. 
\end{equation}
To control the right-hand side of~\eqref{e:stimaesse1} we use the Sobolev-Gagliardo-Niremberg Inequality
$$
  \| \d_x u \|_{L^\infty} \leq \sqrt{2} \| \d_x u \|^{1/2}_{L^2}\| \d_{xx} u \|^{1/2}_{L^2} 
$$
and recall that both $u^\eps_1$ and $u^\eps_2$ belong to the regularity class~\eqref{e:regularity}. We also have 
\begin{equation}\label{e:stimaesse2} \begin{split}
 S_2 &: = \eps^2 \int_0^T \! \!  \int_{\R_+} [u_1^\eps - u_2^\eps]  \d_{xxx} [ u_1^\eps- u_2^\eps] dx dt = - \eps^2 \int_0^T \! \!  \int_{\R_+} \d_x [u_1^\eps - u_2^\eps]  \d_{xx} [ u_1^\eps- u_2^\eps] dx dt \\
 & = \frac{\eps^2}{2} \int_0^T \big[ \d_x [u_1^\eps - u_2^\eps] \big]^2 (t, 0) dt. 
 \end{split}
\end{equation}
By plugging~\eqref{e:stimaesse1} and~\eqref{e:stimaesse2} into~\eqref{e:uniqueness} and applying the Gronwall Lemma we conclude that $\| [u^\eps_1 - u^\eps_2] (T, \cdot) \|^2_{L^2} \equiv 0$ for every $T \in [0, T_0]$, whence $u^\eps_1 \equiv u^\eps_2$.

\subsection{Existence result for the approximating system}
\label{s:app}
This paragraph aims at establishing the following. 
\begin{proposition} \label{p:exapp}
    Under the same hypotheses as in the statement of Theorem~\ref{t:main}, the Cauchy problem obtained by coupling~\eqref{eq:w_nu} with the data 
    \begin{equation} \label{e:datawnu2}
    w ^{\eps \nu} (0, x) =0, \quad w^{\eps \nu}(t, 0) =0
    \end{equation}
    has a solution $w^{\eps \nu} \in H^3(]0, T_0[ \times \R_+).$  
\end{proposition}
\subsubsection{Change of variables}
In this paragraph we establish local-in-time existence for~\eqref{eq:w_nu}
by closely following~\cite{BonaWinther}. Towards this end, we couple~\eqref{eq:w_nu} with the initial and boundary data
\begin{equation}
    \label{e:datawnu}
        w^{\eps \nu}(0, x) = w_0 (x), \quad w^{\eps \nu}(t, 0) =0,
\end{equation}
where $w_0 \in H^3(\R_+)$ satisfies $w_0(0)=0$. We introduce the change of variables
\begin{equation} \label{e:changeofvar}
     w^{\eps \nu}(t, x) = v^{\eps \nu} (t, \eps^2 t + \nu x), 
\end{equation}
and from~\eqref{eq:w_nu} we deduce 
\begin{equation}\label{eq:v_nu}
\left\{
\begin{aligned}
&\d_t v^{\eps \nu}- \nu^3 \d_{txx} v^{\eps \nu} = s(t, x), \\
& v^{\eps \nu} (x= \eps^2 t) =0, \qquad v^{\eps \nu}(t=0) = w_0^\nu,
\end{aligned}
\right.
\end{equation}
provided 
\begin{equation}\label{e:source}
w_0^\nu(x) = w_0 (x/\nu), \qquad s(t, x) = - \eps^2 \d_x v^{\eps \nu} - \nu \d_x[f'(u^\eps_a)v^{\eps \nu}+g(u^\eps_a, v^{\eps \nu})(v^{\eps \nu})^2] - \tilde{\mathcal E}^{inn}-\tilde{\mathcal E}^{b},
\end{equation}
and $\tilde{\mathcal E}^{inn}$ and $\tilde{\mathcal E}^{b}$ are defined by rescaling 
$\mathcal E^{inn}$ and $\mathcal E^{b}$, respectively. More precisely, 
\begin{equation*}
    \tilde{\mathcal E}^{inn} (t, x): = \mathcal E^{inn}
    \left( t, \frac{x - \eps^2 t}{\nu}\right), \quad 
     \tilde{\mathcal E}^{b} (t, x): = \mathcal E^{b}
    \left( t, \frac{x - \eps^2 t}{\nu}\right).
\end{equation*}
We integrate the equation at the first line of~\eqref{eq:v_nu} in time and arrive at the ordinary differential equation
$$
    v^{\eps \nu}- \nu^3 \d_{xx} v^{\eps \nu} = w_0^\nu  - \nu^3 \d_{xx} w_0^\nu+ \int_0^t s(t', \cdot) d t',
$$
which can be solved explicitly using the initial datum in~\eqref{eq:w_nu}. We get 
\begin{equation}\label{e:fixedpt}
\begin{split}
      v^{\eps \nu} (t, x) & = \frac{1}{2 \nu^{3/2}}\int_{\eps^2 t}^{+ \infty} \left[ \exp \big[ - |x-y| \nu^{-3/2}\big]- \exp \big[ - (x+y - 2 \eps^2 t) \nu^{-3/2} \big] \right]\\
      & \qquad \times 
   \left[ w_0^\nu -   \nu^3 \d_{yy} w_0^\nu+ \int_0^t s (t', y) d t' \right] dy .  
   \end{split}
\end{equation} 
We now perform some computations to write the above formula in a more convenient way. First, using the explicit expression of $s$ in~\eqref{e:source}, the Integration by Parts Formula and the boundary datum $v^{\eps \nu}(t, \eps^2 t) =0$, we get 
\begin{equation*}
\begin{split}
      & - \frac{1}{2 \nu^{3/2}} \int_{\eps^2 t}^{+ \infty} \left[ \exp \big[ - |x-y| \nu^{-3/2}\big]- 
        \exp \big[ - (x+y - 2 \eps^2 t) \nu^{-3/2} \big] \right] \\ & \quad \times  \int_0^t \eps^2 \d_y v^{\eps \nu} + \nu \d_y  
      [f'(u^\eps_a)v^{\eps \nu}+g(u^\eps_a, v^{\eps \nu})(v^{\eps \nu})^2] dt' \ dy   \\ &
      = \frac{1}{2 \nu^3}\int_{\eps^2 t}^{+ \infty} \left[ \mathrm{sign} [x-y] \exp \big[ - |x-y| \nu^{-3/2}\big]+ \exp \big[ - (x+y - 2 \eps^2 t) \nu^{-3/2} \big] \right] \\ & \quad \times \int_0^t \  v^{\eps \nu} + \nu [
      f'(u^\eps_a)v^{\eps \nu}+g(u^\eps_a, v^{\eps \nu})(v^{\eps \nu})^2] dt' \ dy.
\end{split}
\end{equation*}
Applying the Integration by Parts formula twice, and using the fact that both $w_0^\nu$ and the exponential terms in~\eqref{e:fixedpt} vanish at $x=\eps^2t$, we similarly obtain 
\begin{equation*}
\begin{split}
       & - \frac{1}{2 \nu^{3/2}} \int_{\eps^2 t}^{+ \infty} \left[ \exp \big[ - |x-y| \nu^{-3/2}\big]- 
        \exp \big[ - (x+y - 2 \eps^2 t) \nu^{-3/2} \big] \right] \nu^3 \d_{yy} w_0^\nu \ dy   \\ &
      = \frac{1}{2}\int_{\eps^2 t}^{+ \infty} \left[ \mathrm{sign} [x-y] \exp \big[ - |x-y| \nu^{-3/2}\big]+ \exp \big[ - (x+y - 2 \eps^2 t) \nu^{-3/2} \big] \right] \d_{y} w_0^\nu \ dy \\
      & = w_0^\nu (x)  -  \frac{1}{2 \nu^{3/2}} \int_{\eps^2 t}^{+ \infty} \left[ \exp \big[ - |x-y| \nu^{-3/2}\big]- 
        \exp \big[ - (x+y - 2 \eps^2 t) \nu^{-3/2} \big] \right]  w_0^\nu \ dy.  
\end{split}
\end{equation*}
Combining the above formulas with~\eqref{e:fixedpt} we eventually arrive at 
\begin{equation}\label{e:fixedpt2}
\begin{split}
      v^{\eps \nu} (t, x) & = w_0^\nu (x) - \frac{1}{2 \nu^{3/2}}\int_{\eps^2 t}^{+ \infty} \left[ \exp \big[ - |x-y| \nu^{-3/2}\big]- \exp \big[ - (x+y - 2 \eps^2 t) \nu^{-3/2} \big] \right]\\
      & \qquad \times 
   \left[\int_0^t [ \tilde{\mathcal E}^{inn}+\tilde{\mathcal E}^{b}] d\tau \right] dy  \\
   & + \frac{1}{2 \nu^3}\int_{\eps^2 t}^{+ \infty} \left[ \mathrm{sign} [x-y] \exp \big[ - |x-y| \nu^{-3/2}\big]+ \exp \big[ - (x+y - 2 \eps^2 t) \nu^{-3/2} \big] \right] \\ & \quad \times \int_0^t \Big[  v^{\eps \nu} + \nu [
      f'(u^\eps_a)v^{\eps \nu}+g(u^\eps_a, v^{\eps \nu})(v^{\eps \nu})^2] \Big] dt' dy.
   \end{split}
\end{equation} 
We are now in a position to establish the local-in-time well-posedness of the Cauchy problem~\eqref{eq:w_nu},\eqref{e:datawnu}.
\begin{lemma}\label{l:lerp}
        Under the same assumptions as in the statement of Theorem~\ref{t:main}, let $g$ be the same as in~\eqref{eq:g}, $ w^\nu_0 \in  C^0 (\R_+) \cap L^2 (\R_+)$ and 
        $$
           \Omega_\tau : = \{ (t, x): \; t \in [0, \tau], \; x \ge \eps^2 t \}. 
        $$
        Then there is $\hat T$ such that, for every $\tau \leq \hat T$, there is a (unique) solution $v^{\eps \nu} \in C^0(\Omega_{\tau}) \cap L^2 (\Omega_\tau)$  of the fixed point problem~\eqref{e:fixedpt}. The value 
        $ \hat T$ only depends on the following quantities:  $\nu, \|\mathcal E^{inn}\|_{C^0}, \| \mathcal E^{b}\|_{C^0},  \|f\|_{C^0}, \|g\|_{C^0}, \| w_0^\nu \|_{C^0}, \| w_0^\nu \|_{L^2}$.  Also, if $w_0^\nu \in H^3 (\R_+)$ then $v^{\eps \nu} \in H^3 (\Omega_{\hat T})$.
\end{lemma}
\begin{proof}
The proof closely follows the argument in~\cite[\S3]{BonaWinther}, so we only provide a sketch. \\
{\bf Step 1:} we fix $\tau>0$ to be determined in the following and $M: =2 \| w_0^\nu \|_{C^0} +  2 \| w_0^\nu \|_{L^2} $. We term $B_M(0)$ the closed ball of radius $M$ and center at $0$ in $C^0_b \cap L^2  (\Omega_\tau)$, equipped with the norm $\| \cdot \|_{C^0} + \| \cdot \|_{L^2}$. We define the map $\mathcal T: B_M (0) \to  C^0 \cap L^2 (\Omega_\tau)$ by setting 
\begin{equation}\label{e:T}
\begin{split}
      [\mathcal T z] &(t, x)  =  w_0^\nu (x) - \frac{1}{2 \nu^{3/2}}\int_{\eps^2 t}^{+ \infty} \left[ \exp \big[ - |x-y| \nu^{-3/2}\big]- \exp \big[ - (x+y - 2 \eps^2 t) \nu^{-3/2} \big] \right]\\
      & \qquad \times 
   \left[\int_0^t [ \tilde{\mathcal E}^{inn}+\tilde{\mathcal E}^{b}] d\tau \right] dy  \\
   & + \frac{1}{2 \nu^3}\int_{\eps^2 t}^{+ \infty} \left[ \mathrm{sign} [x-y] \exp \big[ - |x-y| \nu^{-3/2}\big]+ \exp \big[ - (x+y - 2 \eps^2 t) \nu^{-3/2} \big] \right] \\ & \quad \times \int_0^t \Big[  z + \nu [
      f'(u^\eps_a)v^{\eps \nu}+g(u^\eps_a, z)z^2] \Big] dt' dy
   \end{split}
\end{equation} 
and point out that, due to the decay of the exponential functions in the above formula, there is $\hat T = \hat T (\nu, \|\mathcal E^{inn}\|_{C^0}, \| \mathcal E^{b}\|_{C^0},  \|f\|_{C^0}, \|g\|_{C^0}, M)$ such that for every $\tau \leq \hat T$ we have that $\mathcal T$ attains values in $B_M(0)$ and is moreover a contraction (that is, a Lipschitz continuous function with Lipschitz constant smaller than $1$). By the Contraction Map Theorem, we conclude that there is a unique function $v^{\eps \nu} \in B_M(0)$ satisfying~\eqref{e:fixedpt2}. \\
{\bf Step 2:} we assume that $w_0^\nu \in C^1 (\R_+)$ and compute the $x$ derivative~\eqref{e:fixedpt2}, obtaining
\begin{equation}\
\begin{split}
      \d_x v^{\eps \nu}&  (t, x)  =  \d_x w_0^\nu (x) \\ & - \frac{1}{2 \nu^{3}}\int_{\eps^2 t}^{+ \infty} \left[- \mathrm{sign} [x-y] \exp \big[ - |x-y| \nu^{-3/2}\big] \right. \\ &
      \quad \left. + \exp \big[ - (x+y - 2 \eps^2 t) \nu^{-3/2} \big] \right] \times 
   \left[ \int_0^t [ \tilde{\mathcal E}^{inn}+\tilde{\mathcal E}^{b}] dt' \right] dy  \\
   & - \frac{1}{2 \nu^{9/2}}\int_{\eps^2 t}^{+ \infty} \left[ \exp \big[ - |x-y| \nu^{-3/2}\big]+   \exp \big[ - (x+y - 2 \eps^2 t) \nu^{-3/2} \big] \right] \\ & \quad \times \int_0^t \Big[  v^{\eps \nu} + \nu [
      f'(u^\eps_a)v^{\eps \nu}+g(u^\eps_a, v^{\eps \nu})(v^{\eps \nu})^2] \Big] d t' \ dy \\ &  \quad + \frac{1}{\nu^3} \int_0^t  v^{\eps \nu} (t, x)+ \nu [
      f'(u^\eps_a)v^{\eps \nu}+g(u^\eps_a, v^{\eps \nu})(v^{\eps \nu})^2](t', x) dt'. \phantom{\int}
   \end{split}
\end{equation} 
Given that $\d_x w_0^\nu \in  L^2(\Omega_\tau)$ by assumption, and recalling that $f \in C^4$, the above formula implies that $\d_x v^{\eps \nu} \in L^2 (\Omega_\tau)$. By an analogous argument we obtain $\d_t v^{\eps \nu} \in L^2 (\Omega_\tau)$, whence $v^{\eps \nu} \in H^1 (\Omega_{\tau})$. By iterating the same argument we get $v^{\eps \nu} \in H^3 (\Omega_{\tau})$.
\end{proof}
\subsubsection{Existence on the time interval $[0, T_0]$ for the approximating equation}
Applying Lemma~\ref{l:lerp} with $w_0^\nu \equiv 0$ and going back to the original variables $w^{\eps \nu}$ through~\eqref{e:changeofvar}, we establish local in time existence of a smooth solution of the Cauchy problem~\eqref{eq:w_nu},\eqref{e:datawnu2}. To conclude the proof of Proposition~\ref{p:exapp}, we have to show that the solution can be extended up to the existence time $T_0$ of the smooth solution of~\eqref{eq:burg}. Towards this end, we proceed as follows. If the value $\hat T$ in the statement of Lemma~\ref{l:lerp} satisfies $\hat T \ge T_0$, then there is nothing to prove. Otherwise, we go back to the analysis in \S\ref{s:main} and multiply~\eqref{eq:w_nu} times $\d_t W^{\eps \nu}$, where $W^{\eps \nu}$ is the anti-derivative of $w^{\eps \nu}$, that is, the function defined by~\eqref{eq:W} with $w^\eps$ replaced by $w^{\eps \nu}$. Note that we can repeat the whole analysis in \S\ref{s:main}, the only difference is that we have to control the additional term 
$$
   \nu \int_0^T \! \! \! \!\int_{\R_+}\! \!\!  \d_t W^{\eps \nu} \d_{txx} w^{\eps \nu}\!  = \! - \nu \int_0^T  \! \! \! \! \int_{\R^+}\! \!\! \!
   \d_t w^{\eps \nu} \d_{tx} w^{\eps \nu} dx \! =\!  - \frac{\nu}{2} \int_0^T\! \!  \! \! \int_{\R^+} \! \!\! \!\d_x[ (\d_t w^{\eps \nu})^2 ] dx\!  =\! 0, 
$$
where in the last equality we have used the homogeneous boundary condition in~\eqref{eq:w_nu}. Repeating the analysis in \S\ref{s:main} we conclude that 
\begin{equation} \label{e:bdh1weenu}
      \int_{\R_+} \left[(w^{\eps \nu})^2 + \frac{\eps^2}{2}  (\d_x w^{\eps \nu})^2 \right] (t, x)  \ dx  \leq \hat C \eps^3, 
\end{equation}
for every $t \in [0, \hat T]$, which in turn yields 
\begin{equation}
    \label{e:linftyK}
      \| w^{\eps \nu} (t, \cdot)  \|_{L^\infty} \leq \hat C \eps, \quad \text{for every $t \in [0, \hat T]$}. 
\end{equation}
This provides a uniform bound on the quantities on which the existence time $\hat T$ in the statement of Lemma~\ref{l:lerp} depends and implies that we iteratively apply Lemma~\ref{l:lerp} with $w_0^\nu = w^{\eps \nu}(\hat T, \cdot)$, $w_0^\nu = w^{\eps \nu}(2 \hat T, \cdot)$, etc.,  and establish existence for $w^{\eps \nu}$ on the time interval $[0, T_0]$. 
\subsection{Additional regularity}
\label{s:reg}
\subsubsection{Preliminaries}
We recall~\eqref{eq:w_nu} and that, owing to \eqref{e:V000},~\eqref{e:Vt0a0},~\eqref{eq:ein} and~\eqref{eq:eb}, $\mathcal E^{inn}(t=0)\equiv \eps^2 \d_{xxx} u_{in}$ and  $\mathcal E^{b}(t=0) = 0$, respectively.  We conclude that the function $\d_t w^{\eps \nu}(0, \cdot)$ solves at $t=0$ the elliptic problem
$$
   z \in H^1_0(\R_+), \quad z - \nu \d_{xx} z + \eps^2 \d_{xxx} u_{in}=0, 
$$
whence 
\begin{equation} \label{e:stimawt0}
   \| \d_t w^{\eps \nu} (0, \cdot) \|_{L^2 } \leq \eps^2 \| \d_{xxx} u_{in} 
   \|_{L^2} \stackrel{\eqref{e:decayV0t}}{\leq} C \eps^2 .
\end{equation}
To obtain a bound on $ \| \d_{xt} w(0, \cdot) \|_{L^2 } $ which does not depend on $\nu$, we use the equation for $w^{\eps \nu}$ to deduce 
\begin{equation*}
    \begin{split}
        \nu \d_{xxt }w^{\eps \nu }(0, 0)& =[ \mathcal E^{inn}+\mathcal E^{b}] (0, 0) = \eps^2 \d_{xxx} u_{in} (0) \stackrel{\eqref{e:nonmipiace}}{=}          0,
    \end{split}
\end{equation*}
whence $h= \d_{xt} w^{\eps \nu} (0, \cdot)$ is the solution of the elliptic problem
$$
  h \in H^1 (\R_+), \quad h'(0)=0, \quad h - \nu \d_{xx} h + \eps^2 \d_{xxxx} u_{in} =0, 
$$
so 
\begin{equation} \label{e:stimawtx0}
   \| \d_{tx} w^{\eps \nu} (0, \cdot) \|_{L^2 } \leq \eps^2 \| \d_{xxxx} u_{in}  
   \|_{L^2} \leq \hat C \eps^2.  
\end{equation}
\subsubsection{Continuous induction argument}
We define $T^{\ast \ast}$ by setting 
\begin{equation} \label{eq:Tastast}
      T^{\ast \ast} : = \sup \left\{ \! t \in [0, T_0]: \! \int_{\R_+} 
      \! \! \left[(\d_t w^{\eps \nu})^2 + \frac{\eps^2}{2}  (\d_{xt} w^{\eps \nu})^2 \right](\tau, x)  \ dx  \leq A \eps^2 \; \forall \, \tau \in [0, t] \right\}
\end{equation}
for a suitable constant $A$, to be determined in the following. 
We want to apply a continuous induction argument. Note that, as long as $t \leq T^{\ast \ast}$, we have 
\begin{equation}
\label{e:wtlinfty}
   \| \d_t w^{\eps \nu} (t, \cdot) \|_{L^\infty} \leq 
   \sqrt{2} \| \d_t w^{\eps \nu} (t, \cdot) \|^{1/2}_{L^2}
   \| \d_{tx} w^{\eps \nu} (t, \cdot) \|^{1/2}_{L^2} \leq 2^{3/4}  \sqrt{A \eps}.
\end{equation}
Next, we set 
\begin{equation} \label{e:Zeta}
    Z^{\eps \nu}(t, x) : = \int_0^x \d_{t} w^{\eps \nu} (t, y) dy
\end{equation}
and derive the equation at the first line of~\eqref{eq:w_nu} with respect to the $t$ variable to obtain
\begin{equation}
    \label{e:KdVt}
   \d_{tt} w^{\eps \nu}+\d_{xt}[f'(u^\eps_a)w^{\eps \nu}+g(u^\eps_a, w^{\eps \nu})(w^{\eps \nu})^2]+\eps^2\d_{xxxt}w^{\eps \nu} - \nu \d_{ttxx} w^{\eps \nu} + \d_t [\mathcal E^{inn}+\mathcal E^{b}] =0.
\end{equation}
We fix $t\leq T^{\ast \ast}$, multiply the above equation times $\d_t Z^{\eps \nu}$ and integrate on $[0, T] \times \R_+$ to obtain 
$$
J_1 + \dots + J_6 =0,
$$
where the terms $J_1, \dots, J_6$ are defined in a similar way as the terms $I_1, \dots, I_6$ in \S\ref{s:main}. In Appendix~\ref{a:J16} we provide their precise definition and show that an argument fairly similar to the one in \S\ref{s:main} yields  
\begin{equation} \label{e:regfinale}
       \frac{c}{4}\| \d_t w^{\eps \nu} (T, \cdot) \|^2_{L^2} + \int_{\R_+} \! \! \frac{\eps^2}{2}  [\d_{xt} w^{\eps \nu}]^2 (T, x)  \ dx  \leq \hat C \eps^2 
      + \hat C[\sqrt{A \eps} +1 ]\int_0^T \| \d_t w^{\eps \nu} (t, \cdot) \|^2_{L^2} dt, 
    \end{equation}
provided $\eps$ is smaller than a threshold only depending on $\hat C$. 
Owing to the Gr\"onwall Lemma, this implies 
$$
   \| \d_t w^{\eps \nu} (T, \cdot) \|^2_{L^2} \leq \hat C \eps^2 \exp [2 \hat C T],
$$
provided $A \eps \leq 1$. Using the arbitrariness of $T$ and again~\eqref{e:regfinale}, this also implies $\eps^2 \| \d_{xt}  w^{\eps \nu} (T, \cdot)\|_{L^2} \leq \hat{C} \eps^2 \exp [2 \hat C T]$. By using a classical continuous induction argument, we conclude that the value $T^{\ast \ast}$ defined by~\eqref{eq:Tastast} coincides with $T_0$ provided $A \ge \hat C  \exp [2 \hat C T_0]$. This, in particular, implies
\begin{equation}
\label{e:hoe}
\int_{\R_+} 
      \! \! \left[(\d_t w^{\eps \nu})^2 + \frac{\eps^2}{2}  (\d_{xt} w^{\eps \nu})^2 \right](t, x)  \ dx  \leq \hat C \eps^2, \quad \text{for every $t \in [0, T_0]$. }
\end{equation}

\subsubsection{Conclusion}
Combining~\eqref{e:hoe} with~\eqref{eq:w_nu} we get 
\begin{equation*}
    \begin{split}
        \eps^2 \| \d_{xxx} w^{\eps \nu}(t, \cdot)\|_{H^{-1}} 
        & \leq \| \d_t w (t, \cdot) \|_{L^2} + 
        \| [f'(u_a) w + g w^2 ] (t, \cdot) \|_{L^2}+ \nu \| \d_{tx} w \|_{L^2} \\
        & \quad + \| \mathcal E^{inn} + \mathcal{E}^b \|_{L^2}  \leq \hat C \sqrt{\eps}  ,   
    \end{split}
\end{equation*}
provided $\nu \leq \sqrt{\eps}$. 
We now recall the interpolation inequality\footnote{Inequality~\eqref{e:interpolation} is well known. For the sake of completeness, we provide the proof in the appendix.}
\begin{equation}\label{e:interpolation}
   \| \d_{xx} w^{\eps \nu} (t, \cdot) \|_{L^2} \leq C [\| \d_{xxx} w^{\eps \nu} (t, \cdot) \|_{H^{-1}} + \| \d_{x} w^{\eps \nu} (t, \cdot) \|_{L^2}], 
 \end{equation}
 which combined with~\eqref{e:bdh1weenu} eventually implies that the approximating sequence $w^{\eps \nu}$ satisfies  
 \begin{equation} \label{e:stimeunifnu}
     \| \d_t w^{\eps \nu} (t, \cdot) \|^2_{L^2} \leq \hat C \eps^2, 
     \quad 
     \|  w^{\eps \nu} (t, \cdot) \|_{H^2}  \leq \hat C \eps^{-3/2}, \quad \text{for every $t \in [0, T_0]$} 
 \end{equation}
 and provided $\nu$ is sufficiently small. 
 \subsection{Passage to the limit}\label{s:limit}
 We use~\eqref{e:stimeunifnu} and apply the Aubin-Lions Lemma: for instance, we can use formula (0.5) at page 67 in~\cite{Simon} with $p=2$, $B= H^1 (I)$ $X=H^2 (I)$, where $I \subseteq \R_+$ is any bounded interval. We conclude that, up to subsequences (that we do not relabel), $w^{\eps \nu}$ converges as $\nu \to 0^+$ to some limit function $w^\eps$ in $C^0([0, T_0], H^1_{\mathrm{loc}}(\R_+)$. Note the limit function satisfies the initial and boundary conditions in~\eqref{eq:w_eps} and is a distributional solution of the equation at the first line of~\eqref{eq:w_eps}. By extracting (if needed) a further subsequence, we can assume that 
the regularity estimates~\eqref{e:stimeunifnu} persist in the the vanishing $\nu$ limit and hence that $w^\eps$ satisfies 
\begin{equation*}
    \begin{split}
        \eps^2 \| \d_{xxx} w^\eps(t, \cdot)\|_{L^2} 
        & \leq \| \d_t w (t, \cdot) \|_{L^2} + 
        \| \d_x [f'(u_a) w^\eps + g [w^\eps]^2 ] (t, \cdot) \|_{L^2} + \| \mathcal E^{inn} + \mathcal{E}^b \|_{L^2}  \\ & \leq \hat C \quad \text{for every $t \in [0, T_0]$}    
    \end{split}
\end{equation*}
which yields $w^\eps \in L^\infty ([0, T_0]; H^3 (\R_+)).$ Since $w^\eps$ in $C^0([0, T_0], H^1_{\mathrm{loc}}(\R_+))$, by interpolation we conclude that $w^\eps \in C^0([0, T_0], H^2 (\R_+))$. 
\begin{remark}
Note that our construction yields a bound on $\| \d_{xxx} w^\eps(t, \cdot)\|_{L^2} $ which deteriorates in the vanishing $\eps$ limit. However, in \S\ref{s:main} we only use the $H^3$ regularity of $w^\eps$ to ensure that all the integrals we write are well defined, and the final estimate~\eqref{eq:main} provides a bound that is uniform in $\eps$ for the $H^1$ norm only. 
\end{remark}
 
\appendix
\section{Proof of Lemma~\ref{l:bl}}
We first outline the rationale underpinning the proof. Under suitable regularity assumptions, the equation in the first line of~\eqref{eq:bl_eq2} implies 
\begin{equation}\label{e:kappa1}
    f(\bar u^0+V)-f(\bar u^0)+\d_{yy}V= k_1
\end{equation}
for a suitable constant $k_1 \in \R$. The asymptotic condition in the second line of~\eqref{eq:bl_eq2} yields
\begin{equation} \label{e:kappa12}
      k_1 = \lim_{y \to + \infty} \d_{yy} V(y) =0,     
\end{equation}
where to establish the last equality we have again used the asymptotic condition in the second line of~\eqref{eq:bl_eq2}. 
We set 
\begin{equation} \label{e:eFFe}
     F(\bar u^0, V) : = \int_0^{V} [ f(\bar u^0+ \xi)-f(\bar u^0)] d\xi \implies \d_2 F(\bar u^0, V) =
  f(\bar u^0+V)-f(\bar u^0),
\end{equation}
where $\d_2$ denotes the partial derivative with respect to the variable $V$. Next, we point out that 
\begin{equation}\label{e:F}
  F(0) =0, \quad F'(0)=0, \quad \d_{22} F(\bar u^0, V) = f'(\bar u^0+ V) \stackrel{\eqref{eq:c}}{\leq} -c,
      \end{equation}
that is the map $V \mapsto F(\bar u^0, V)$ loosely speaking behaves like a concave-down parabola with the vertex at $V=0$.       
Next, we multiply~\eqref{e:kappa1} times $\d_y V$ to arrive at 
\begin{equation*}
    \d_y \left[ F(\bar u^0, V) + \frac12 [\d_y V]^2 \right] = 0\implies 
    F(\bar u^0, V) + \frac12 [\d_y V]^2  = k_2
\end{equation*}
for a suitable constant $k_2 \in \R$. Using the asymptotic condition in~\eqref{eq:bl_eq2} we conclude that  
$$
  k_2 = \lim_{y \to + \infty }  \frac12 [\d_y V]^2 =0.
$$
Wrapping up, we are led to study
\begin{equation} \label{e:veraeq}
      F(\bar u^0, V) + \frac12 [\d_y V]^2 =0.
\end{equation}
{\bf Step 1:} we establish existence and~\eqref{e:decay}. Note that~\eqref{e:F} implies
\begin{equation}\label{e:F2}
     F \leq 0, \quad \sqrt{- F} \in W^{1 \infty}_{\mathrm{loc}} (\R), \quad F(\bar u^0, 0) =0,  \quad F(\bar u^0, V) =0 \Rightarrow  V=0
     \quad 
     \end{equation}
and separately consider the cases $\bar V \leq 0$ and $\bar V>0$.  
If $\bar V \leq 0$, we consider the Cauchy problem 
\begin{equation} \label{e:caupb}
      \left\{
           \begin{array}{ll}
                   \d_y V = \sqrt{-2 F(\bar u^0, V)} \\
                   V(y=0) = \bar V,  
           \end{array}
      \right.
\end{equation}
which satisfies the assumptions of the Cauchy-Lipschitz-Picard–Lindel\"of Theorem for ODEs due to the second condition in~\eqref{e:F2}. Note that the unique solution of~\eqref{e:caupb} is a monotone non decreasing function, and as such converges to a limit as $y\to + \infty$. If finite, the limit must be $0$, the unique equilibrium point of $F$ due to the last condition in~\eqref{e:F2}. Assume by contradiction that the limit is $+\infty$, then by continuity $V$ must attain the value $0$ at some point. Since the solution of~\eqref{e:caupb} is unique, this implies that $V(y)=0$ for every $y$, and contradicts the assumption that $V$ converges to $+\infty$. 

We are left to  establish~\eqref{e:decay}. Towards this end, we point out that 
\begin{equation*}\begin{split}
    F(\bar u^0, V)&\stackrel{F(\bar u^0, 0)=0}{=} - \int_{V}^0 \d_2 F(\bar u^0, \xi) d \xi \stackrel{F'(0)=0}{=}
   \int_{V}^0 \int_{\xi}^0 \d_{22} F (\bar u^0, \eta) d \eta d \xi  \\ & \stackrel{\eqref{e:F}}{=}
   \int_{V}^0 \int_{\xi}^0 f'(\bar u^0 + \eta) d \eta d \xi \stackrel{\eqref{eq:c}}{\leq} - \frac{c}{2} V^2, 
   \end{split}
\end{equation*}
which implies 
\begin{equation*}
     \sqrt{-2 F (V)} \ge  |V| \sqrt{c}\stackrel{V \leq 0}{=}
     - V \sqrt{c}
\end{equation*}
and by the Comparison Theorem for ODEs applied to~\eqref{e:caupb} this implies 
$V(y) \ge - |\bar V| \exp[- \sqrt{c} y]$, which combined with the inequality $V \leq 0$ leads to~\eqref{e:decay}. 
The analysis of the case $\bar V >0$ is entirely similar and is therefore omitted. \\
{\bf Step 2:} uniqueness. By contradiction, we assume that there are two different solutions $V_{01}$ and $V_{02}$ and we set $Z:= V_{01}-V_{02}$. Note that $Z(0)=0$, 
$\lim_{y \to + \infty} Z(y)=0$ and that  
\begin{equation} \label{e:Ftilde}
      f (\bar u^0 + V_{01}) - f(\bar u^0 + V_{02})  +\d_{yy} Z = k_3 =0,
\end{equation}
where to establish the last equality we have used the asymptotic conditions on $V_{01}$ and $V_{02}$. We write the Taylor expansion with Lagrange remainder
$$
f (\bar u^0 + V_{01}) - f(\bar u^0 + V_{02})= f' (\eta (V_1, V_2)) Z
$$
and  multiply~\eqref{e:Ftilde} times $Z$ and integrate over $\R_+$. Integrating by parts  and using the boundary condition $Z(0)=0$, we get 
$$
- \int_{\R_+} [\d_y Z]^2 =  \int_{\R_+} \underbrace{- f' (\eta (V_1, V_2))}_{\ge c >0} Z^2,  
$$ 
which yields $Z \equiv 0$ and concludes the proof of Lemma~\ref{l:bl}.

\section{Estimates on $J_1, \dots, J_6$} \label{a:J16}
In this appendix we control the terms $J_1, \dots, J_6$ obtained multiplying~\eqref{e:KdVt} times $\d_t Z^{\eps \nu}$, where $Z^{\eps \nu}$ is defined in~\eqref{e:Zeta}, and integrating in space and time. We have 
\begin{equation*}
    \begin{split}
J_1= \int_0^{T} \! \! \! \! \int_{\R_+} \d_t Z^{\eps \nu} \d_{tt} w^{\eps \nu} (t, x) dx dt 
= \int_0^{T}  \! \! \! \! \int_{\R_+} \d_t Z^{\eps \nu} \d_{tx}  Z^{\eps \nu}(t, x)   dx dt 
\stackrel{\d_t Z^{\eps \nu}(t, 0)=0}{=}
0      
    \end{split}
\end{equation*}
and (integrating by parts with respect to the $x$ variable) 
\begin{equation*}
    \begin{split}
J_2 & = \int_0^{T} \! \! \! \! \int_{\R_+} \d_t Z^{\eps \nu} \d_{tx} [f'(u_a) w ]dx dt 
\stackrel{\d_t Z^{\eps \nu}(t, 0)=0}{=}-  \int_0^{T}  \! \! \! \! \int_{\R_+} \d_{tt} w^{\eps \nu} \d_t [f'(u_a) w^{\eps \nu} ]dx dt \\
& = \underbrace{ - \int_0^{T} \! \! \! \! \int_{\R_+} \d_{tt} w^{\eps \nu} \d_t [f'(u_a)] w^{\eps \nu} dx dt}_{:= J_{21}}  \underbrace{- 
\int_0^{T} \! \! \! \!\int_{\R_+}  f'(u_a)\d_{tt} w^{\eps \nu}  \d_t w^{\eps \nu}  dx dt }_{: = J_{22}}.
    \end{split}
\end{equation*}
Note that integrating by parts with respect to the $t$ variable we get 
\begin{equation*}
    \begin{split}
J_{21} & \stackrel{w^{\eps \nu}(0, \cdot) =0}{=} - 
\int_{\R_+} \d_t w^{\eps \nu}  \d_{t}[f'(u_a)] w^{\eps \nu} (T, x)  dx  + \int_0^{T} \! \! \! \! \int_{\R_+} \d_t w^{\eps \nu}  \d_{tt}[f'(u_a)] w^{\eps \nu}  dx dt  \\
      & \quad \quad +
    \int_0^{T} \! \! \! \! \int_{\R_+} |\d_t w^{\eps \nu}|^2  \d_{t}[f'(u_a)]  dx dt,
    \end{split}
\end{equation*} 
whence, using among other things the Young Inequality,  
\begin{equation*}
    \begin{split}
  |J_{21}| & \stackrel{\eqref{e:decayV0t},\eqref{e:decayV0tt}}{\leq} \hat C \| \d_t w^{\eps \nu}(T, \cdot) \|_{L^2} 
  \| w^{\eps \nu}(T, \cdot) \|_{L^2} + \hat C \int_0^T \| \d_t w^{\eps \nu} (t, \cdot) \|^2_{L^2} dt \\
  & \qquad \qquad + 
  \hat C \int_0^T \|  w^{\eps \nu} (t, \cdot) \|^2_{L^2} dt
  \\
& \stackrel{\eqref{e:bdh1weenu}}{\leq}   \hat C \eps \| \d_t w^{\eps \nu}(T, \cdot) \|^2_{L^2} + \hat C \eps^2 +  \hat C \int_0^T \| \d_t w^{\eps \nu} (t, \cdot) \|^2_{L^2} dt. 
  \end{split}
\end{equation*}
Also,  
\begin{equation*}
    \begin{split}
J_{22} & = - \frac{1}{2}\int_0^{T} \! \! \! \! \int_{\R_+} \d_t \big[ [\d_t w^{\eps \nu} ]^2 \big] f'(u_a) dx dt \\
& =
 \frac{1}{2}\int_0^{T} \! \! \! \! \int_{\R_+} [-f'(u_a)][\d_t w^{\eps \nu} ]^2 (T, x) dx 
 + \frac{1}{2}\int_0^{T} \! \! \! \! \int_{\R_+} f'(u_a)][\d_t w^{\eps \nu} ]^2 (0, x) dx \\ & \quad +
 \frac{1}{2}\int_0^{T} \! \! \! \! \int_{\R_+}  [\d_t w ]^2 \d_t [f'(u_a)] dx dt,
    \end{split}
\end{equation*}
whence 
\begin{equation*}
    \begin{split}
    J_{22} \stackrel{\eqref{eq:c},\eqref{e:stimawt0}}{\ge}
    \frac{c}{2} \| \d_t w^{\eps \nu} (T, \cdot) \|^2_{L^2} - \hat C\eps^4  -
    \hat C \int_0^T \| \d_t w^{\eps \nu} (t, \cdot) \|^2_{L^2} dt. 
 \end{split}
\end{equation*}
We also have 
\begin{equation*}
    \begin{split}
J_3 & = \int_0^{T} \! \! \! \! \int_{\R_+} \d_t Z^{\eps \nu} \d_{tx} [g(u_a, w^{\eps \nu}) [w^{\eps \nu}]^2 ]dx dt \\ & 
\stackrel{\d_t Z^{\eps \nu}(t, 0)=0}{=}-  \int_0^{T}  \! \! \! \! \int_{\R_+} \d_{tt} w^{\eps \nu} \d_t \big[g(u_a, w^{\eps \nu})[ w^{\eps \nu}]^2 \big]dx dt \\
& = \underbrace{ -  \int_0^{T} \! \! \! \! \int_{\R_+} \d_{tt} w^{\eps \nu} \d_t w^{\eps \nu}  \big[ 2 g(u_a, w^{\eps \nu}) w^{\eps \nu} + [w^{\eps \nu}]^2 \d_w g (u_a, w^{\eps \nu})  \big] dx dt}_{:= J_{31}} \\ &  \quad   \underbrace{- 
\int_0^{T} \! \! \! \!\int_{\R_+}  [w^{\eps \nu}]^2 \d_{tt} w^{\eps \nu}  \d_{u_a} g (u_a, w^{\eps \nu}) \d_t u_a   dx dt }_{: = J_{32}}
    \end{split}
\end{equation*}
and (integrating by parts with respect to the $t$ variable) 
\begin{equation*}
    \begin{split}
J_{31} & = - \frac{1}{2} \int_0^{T} \! \! \! \! \int_{\R_+} \d_{t} 
\big[ [\d_t w^{\eps \nu} ]^2 \big]  \big[ 2 g(u_a, w^{\eps \nu}) w^{\eps \nu} + [w^{\eps \nu}]^2 \d_w^{\eps \nu}  g (u_a, w^{\eps \nu})  \big] dx dt \\ & \stackrel{w^{\eps \nu}(0, \cdot)=0}={} - 
\underbrace{\frac{1}{2} \int_{\R_+} 
 [\d_t w^{\eps \nu} ]^2  [ 2 g(u_a, w^{\eps \nu}) w^{\eps \nu} + [w^{\eps \nu}]^2 \d_w  g (u_a, w^{\eps \nu})  ] (T, x) dx }_{: = J_{311}}
  \\
 & \quad +\underbrace{ \frac{1}{2} \int_0^{T} \! \! \! \! \int_{\R_+}  
 [\d_t w^{\eps \nu} ]^2  \d_t [ 2 g(u_a, w^{\eps \nu}) w^{\eps \nu} + [w^{\eps \nu}]^2 \d_w  g (u_a, w^{\eps \nu})  ] (t, x) dx dt}_{: = J_{312}}. 
    \end{split}
\end{equation*}
Note that 
$$
  |J_{311}| \stackrel{\eqref{e:linftyK}}{\leq}
  \hat C \eps \| \d_t w^{\eps \nu} (T, \cdot) \|^2_{L^2}
$$
and that 
\begin{equation*}
    \begin{split}
    |J_{312}| & \leq \hat C \int_0^T \big[ \| \d_t u_a \|_{L^\infty} + \| \d_t w^{\eps \nu} \|_{L^\infty}\big]  \| \d_{t} w^{\eps \nu} (t, \cdot) \|^2_{L^2} dt \\ &
    \stackrel{\eqref{e:decayV0t},\eqref{e:wtlinfty}}{\leq}
     \hat C \int_0^T \big[ \sqrt{A\eps} + 1   \big] \| \d_{t} w^{\eps \nu} (t, \cdot) \|^2_{L^2} dt
\end{split}
\end{equation*}
Integrating again by parts with respect to the $t$ variable we get 
\begin{equation*}
    \begin{split}
J_{32} & \stackrel{w^{\eps \nu}(0, \cdot)=0}{=} \!\!\! - \! \int_{\R_+}  [w^{\eps \nu}]^2 \d_{t} w^{\eps \nu}  \d_{u_a} g (u_a, w^{\eps \nu}) \d_t u_a   (T, x) dt \\ & \qquad +
2\int_0^{T} \! \! \! \!\int_{\R_+}  w^{\eps \nu} [\d_{t} w^{\eps \nu}]^2  \d_{u_a} g (u_a, w^{\eps \nu}) \d_t u_a   dx dt \\
& \quad \quad +
\int_0^{T} \! \! \! \!\int_{\R_+}  [w^{\eps \nu}]^2 \d_{t} w^{\eps \nu} \d_t[ \d_{u_a} g (u_a, w^{\eps \nu}) \d_t u_a]   dx dt,
    \end{split}
\end{equation*}
whence, using the H\"older and the Young Inequalities Inequality,  
\begin{equation*}
    \begin{split}
|J_{32}| & \stackrel{\text{H\"older},\eqref{e:decayV0t},
\eqref{e:linftyK}}{\leq}
      \hat C \eps  \|  w^{\eps \nu} (T, \cdot)\|_{L^2}
      \| \d_t w^{\eps \nu} (T, \cdot)\|_{L^2}
      + \hat C  \eps \int_0^T \| \d_t w^{\eps \nu} (t, \cdot) \|^2_{L^2 }
    dt \\
    & \qquad \qquad + \hat C  \int_0^T \|  w^{\eps \nu} (t, \cdot)\|_{L^2} \| \d_t w^{\eps \nu} (t, \cdot) \|^2_{L^2 }\\
    & 
    \stackrel{\text{Young},\eqref{e:bdh1weenu}}{\leq}
    \hat C \eps  \|  w^{\eps \nu} (T, \cdot)\|_{L^2} + \hat C \eps^2 + 
     \hat C  \eps \int_0^T \| \d_t w^{\eps \nu} (t, \cdot) \|^2_{L^2 }.
    \end{split}
\end{equation*}
We also have 
\begin{equation*}
    \begin{split}
    J_4&= 
\eps^2 \int_0^{T} \! \! \! \! \int_{\R_+} \d_t Z^{\eps \nu} \d_{txxx} w^{\eps \nu} (t, x) dx dt \\ & 
\stackrel{\d_t Z^{\eps \nu}(t, 0)=0}{=} - \eps^2 \int_0^{T}  \! \! \! \! \int_{\R_+} \d_{tt} w^{\eps \nu}  \d_{txx} w^{\eps \nu}(t, x)   dx dt  \\ & \stackrel{\d_{tt}w^{\eps \nu}(t, 0)=0}{=}
   \eps^2 \int_0^{T}  \! \! \! \! \int_{\R_+} \d_{ttx} w^{\eps \nu}  \d_{tx} w^{\eps \nu}(t, x) dx dt= \frac{\eps^2}{2} \int_0^{T}  \! \! \! \! \int_{\R_+} \d_{t} \big[ [\d_{tx} w^{\eps \nu}]^2 \big] dx dt 
   \\ & = \frac{\eps^2}{2} \left[
   \int_{\R_+} [\d_{tx} w^{\eps \nu}]^2(T, x) dx  -
    \int_{\R_+} [\d_{tx} w^{\eps \nu}]^2(0, x) dx 
   \right]
    \end{split}
\end{equation*}
whence 
$$
  J_4 \stackrel{\eqref{e:stimawtx0}}{\ge} \frac{\eps^2}{2} 
   \int_{\R_+} [\d_{tx} w^{\eps \nu}]^2(T, x) dx - \hat C \eps^2
$$
Also, 
\begin{equation*}
    \begin{split}
    J_5&= 
- \nu \int_0^{T} \! \! \! \! \int_{\R_+} \d_t Z^{\eps \nu} \d_{ttxx} w^{\eps \nu} (t, x) dx dt 
\stackrel{\d_t Z^{\eps \nu}(t, 0)=0}{=} \nu \int_0^{T}  \! \! \! \! \int_{\R_+} \d_{tt} w^{\eps \nu}  \d_{ttx} w^{\eps \nu}(t, x)   dx dt \\ 
& = \frac{\nu}{2} \int_0^{T}  \! \! \! \! \int_{\R_+} \d_x \big[   [\d_{tt} w^{\eps \nu}]^2 \big](t, x)    dx dt
\stackrel{\d_{tt}w^{\eps \nu}(t, 0)=0}{=}0. 
    \end{split}
\end{equation*}
Finally, we have 
\begin{equation*}
    \begin{split}
    & J_6= 
 \int_0^{T} \! \! \! \! \int_{\R_+} \d_t Z^{\eps \nu} [\d_t \mathcal E^{inn}+ \d_{t}  \mathcal E^b ] dx dt  \stackrel{\eqref{eq:ein},\eqref{eq:eb}}{=}\eps^2 
   \underbrace{\int_0^{T} \! \! \! \! \int_{\R_+} \d_t Z^{\eps \nu} \d_{txxx} u  dx dt}_{: = I_{61}} \\ & 
   + 
\underbrace{\int_0^{T} \! \! \! \! \int_{\R_+} \d_t Z^{\eps \nu}  \d_{tt} V dx dt}_{: = I_{62}} +
   \underbrace{\int_0^{T} \! \! \! \! \int_{\R_+} \d_t Z^{\eps \nu} \d_{tx} H (u, V, u^0) \big] dx dt}_{: = I_{63}}
    \end{split}
\end{equation*}
provided $H$ is the same as in~\eqref{eq:acca}. We have
\begin{equation*}
    \begin{split}
    J_{61}& \stackrel{Z^{\eps \nu}(t, 0) =0}{=} \eps^2 
 \int_0^{T} \! \! \! \! \int_{\R_+} \d_{tt} w \d_{txx} u dx dt = \eps^2 
 \int_{\R_+} \d_{t} w^{\eps \nu} \d_{txx} u (T, x ) dx
 \\ & - \eps^2 
 \int_{\R_+} \d_{t} w^{\eps \nu}(0, \cdot) \d_{txx} u_{in} ( x ) dx -
 \eps^2 
 \int_0^{T} \! \! \! \! \int_{\R_+} \d_{t} w^{\eps \nu} \d_{ttxx} u dx dt,
    \end{split}
\end{equation*}
whence, using the Young Inequality,  
\begin{equation*}
    \begin{split}
|J_{61}| \stackrel{\eqref{e:stimawt0}}{\leq}
      \hat C \eps^2 \| \d_t w^{\eps \nu} (T, \cdot)\|^2_{L^2} + \hat C \eps^{2}
      + \hat C \eps^2 \int_0^T \| \d_t w^{\eps \nu} (t, \cdot) \|^2_{L^2 }
    dt. \end{split}
\end{equation*}
Also,
\begin{equation*}
    \begin{split}
    J_{62}&= 
 \int_0^{T} \! \! \! \! \int_{\R_+} \d_t \big[ Z^{\eps \nu} \d_{tt} V \big] dx dt 
  - \int_0^{T} \! \! \! \! \int_{\R_+} Z^{\eps \nu}   \d_{ttt} V \ dx dt \\
    & =  \int_{\R_+} Z^{\eps \nu} \d_{tt} V  (T, x) \ dx -  
       \int_{\R_+}  Z^{\eps \nu} \d_{tt} V  (T, x) (0, x) \ dx  - \int_0^{T} \! \! \! \! \int_{\R_+} Z^{\eps \nu}   \d_{ttt} V \ dx dt.
    \end{split}
\end{equation*}
We use the H\"older Inequality to get the estimate 
$$
|Z^{\eps \nu} (t, x)| \leq \sqrt{x} \| \d_t w^{\eps \nu} (t, \cdot) \|_{L^2},
$$
which owing to~\eqref{e:decayV0tt} implies 
$$
  \left| \int_{\R_+} Z^{\eps \nu} \d_{tt} V  (T, x) \ dx \right| \leq C \eps^{3/2} \| \d_t w^{\eps \nu} (T, \cdot) \|_{L^2} 
$$
and similarly for the other terms in $J_{62}$. Due to~\eqref{e:stimawt0} and to the Young Inequality we conclude that 
$$
  |J_{62}|  \leq \hat  C \eps \| \d_t w^{\eps \nu} (T, \cdot) \|^2_{L^2} + \hat C\eps^2 + \hat C \eps \int_0^T \| \d_t w^{\eps \nu} (t, \cdot) \|^2_{L^2} dt . 
$$
We also have 
\begin{equation*}
    \begin{split}
    J_{63}& \stackrel{Z^{\eps \nu}(t, 0)=0}{=} - 
 \int_0^{T} \! \! \! \! \int_{\R_+} \d_{tt} w^{\eps \nu} \d_t H  dxdt \\ & = - \int_{\R_+} \d_{t} w^{\eps \nu} \d_t H (T, x) dx + \int_{\R_+} \d_{t} w^{\eps \nu} \d_t H (0, x) dx +  \int_0^{T} \! \! \! \! \int_{\R_+} \d_{t} w^{\eps \nu} \d_{tt} H  dxdt. 
    \end{split}
\end{equation*}
To control the first term in the above expression we recall the explicit expression of $H$, that is~\eqref{eq:acca}, and then~\eqref{e:decayV},\eqref{e:decayV0t},~\eqref{e:decayV0tt} and~\eqref{e:Lagrange}. We conclude that  
$$
  |\d_{tt} H(t, x)|, |\d_{t} H(t, x)| \leq \hat C x \exp \left[- \frac{\sqrt{c} x}{4 \eps}   \right] \implies \| \d_{tt} H(t, \cdot)\|_{L^2}, \|\d_{t} H(t, x)\|_{L^2} \leq \hat C e^{3/2}
$$
and from this, the H\"older and Young Inequalities   and~\eqref{e:stimawt0} we eventually conclude that 
$$
  |J_{63}| \leq \hat C \eps^2 + \hat C \eps \| \d_t w^{\eps \nu} (T, \cdot) \|^2_{L^2} + \hat C \eps \int_0^T  \| \d_t w^{\eps \nu} (T, \cdot) \|^2_{L^2} (t, \cdot) dt.  
$$
By combining the above expressions, we arrive at~\eqref{e:regfinale} provided $\eps$ is sufficiently small. 
\section{Proof of~\eqref{e:interpolation}}
Given any smooth and compactly supported function $z$ we have 
\begin{equation*}
    \| \d_{xx} z \|_{L^2} = \sup_{\varphi \in L^2\setminus \{ 0\} } \frac{\displaystyle{\int_{\R_+}
    \d_{xx} z (x) \varphi (x) dx}}{\|\varphi \|_{L^2}}. 
\end{equation*}
Given $\varphi \in L^2$, we set 
$$
  \psi (x) : = \int_0^x \exp [-(x-y)] \varphi (y) dy, 
$$
which satisfies 
$$
   \psi \in H^1_0 (\R_+), \quad \psi + \psi' = \varphi, \quad \| \psi \|_{L^2} \leq \| \varphi \|_{L^2}, \quad \| \psi' \|_{L^2} \leq 2 \| \varphi \|_{L^2}.
$$
To establish the $L^2$ bound on $\psi$, we write $\psi = \varphi {1}_{\R_+} \ast \eta$, with $\eta (x) = \exp[-x] {1}_{\R_+}$, and apply the Young Theorem on convolution. We then have 
\begin{equation}
    \begin{split}
    \int_{\R_+} \d_{xx} z (x) \varphi(x) dx & =   \int_{\R_+} \d_{xx} z (x) [\psi(x) + \psi'(x) ] dx \\
    & = - \int_{\R_+} [ \d_{x} z (x)\psi'(x) dx + \d_{xxx} z (x)\psi(x) ] dx \\ & \leq \| \d_{x} z \|_{L^2} \| \psi' \|_{L^2 } 
    + \| \d_{xxx} z \|_{H^{-1} }\| \psi \|_{H^1 } \phantom{\int}
    \\ & \leq 3  \big[\| \d_{xxx} z  \|_{H^{-1}} + \| \d_{x} z  \|_{L^2} \big] \| \varphi \|_{L^2}, \phantom{\int}
    \end{split}
\end{equation}
which by density yields~\eqref{e:interpolation}. 
\section*{Acknowledgments}
PA and LVS are members of the GNAMPA group of INdAM. PA and LVS  acknowledge financial support from the Italian Ministry of University and Research (MUR) through the Project 2022YXWSLR \emph{Boundary analysis for dispersive and viscous fluids}. 
PA also acknowledges partial support by the Italian Ministry of University and Research (MUR) through the Excellence Department Project awarded to GSSI, CUP
D13C22003740001.
This work was initiated when LVS was visiting GSSI: its kind hospitality is gratefully acknowledged. 
\bibliographystyle{plain}
\bibliography{PPL}
\end{document}